\documentclass[a4paper,10pt,fleqn]{article}
\usepackage{amsmath}
\usepackage{amssymb}
\usepackage{theorem}
\usepackage{euscript}

\newtheorem{theorem}{Theorem}[section]
\newtheorem{lemma}{Lemma}[section]
\newtheorem{corollary}{Corollary}[section]
\numberwithin{equation}{section}
\title{\sffamily High-dimensional periodic sampling on Smolyak grids \\ based on B-spline quasi-interpolation
\author{ 
Dinh D\~ung \\[5mm]
Vietnam National University, Hanoi, Information Technology Institute \\
144 Xuan Thuy, Cau Giay, Hanoi, Vietnam\\
{\ttfamily dinhzung@gmail.com}\\[4mm]}}
\date{\ttfamily December 20, 2015 --  Version 2.6}
\def\II{{\mathbb I}}

\def\ZZ{{\mathbb Z}}
\def\NN{{\mathbb N}}
\def\RR{{\mathbb R}}

\def\Pp{{\mathcal P}}

\def\Ff{{\mathcal F}}
\def\TT{{\mathbb T}}
\def\IId{{\mathbb I}^d}

\def\NNd{{\mathbb N}^d}
\def\RRd{{\mathbb R}^d}
\def\TTd{{\mathbb T}^d}

\def\ZZd{{\mathbb Z}^d}
\def\RRdp{{\mathbb R}^d_+}
\def\ZZdp{{\mathbb Z}^d_+}

\def\ZZdpu{{\mathbb Z}^d_+(u)}

\def\Pp{{\mathcal P}}

\def\Ff{{\mathcal F}}

\def\Ha{H^\alpha_\infty(\TTd)}
\def\Ua{U^\alpha_\infty}

\def\Uas{\mathring{U}^{\alpha}_\infty}

\def\Uan{U^{\alpha,\nu}_\infty}
\def\Hau{H^\alpha_\infty(u)}
\def\Int{\operatorname{Int}}
\def\supp{\operatorname{supp}}

\begin{document}
\maketitle

\begin{abstract}
We constructed linear algorithms of sampling recovery and cubature formulas on Smolyak grids 
parametrized by $m \in \NN$ of periodic $d$-variate functions having Lipschitz-H\"older mixed smoothness $\alpha > 0$ based on B-spline quasi-interpolation, and  studied their optimality.  We established lower estimates 
(for $\alpha \le 2$) and upper bounds of the error of the optimal sampling recovery and the optimal integration on Smolyak grids, explicit in 
$d$, $m$ and the number $\nu$ of active variables of functions when $d$ and $m$ may be large.


  
\end{abstract}

\section{Introduction}


We are interested in sampling recovery and cubature of functions on $\RRd$ $1$-periodic at each variable. It is convenient to consider them as functions defined in the $d$-torus $\TTd = [0,1]^d$ which is defined as the cross product of
$d$ copies of the interval $[0,1]$ with the identification of the end points. To avoid confusion, we use the notation $\IId$ to denote the standard unit $d$-cube $[0,1]^d$. 

For $m \in \NN$, the well known periodic Smolyak grid of points $G^d(m) \subset \TTd$ is defined as
\begin{equation}  \nonumber
G^d(m)
:= \
\{\xi = 2^{-k} s: k \in \NNd, \ |k|_1 = m, \ s \in I^d(k)\},
\end{equation}
where  $I^d(k) := \{s \in \ZZdp: s_j = 0,1,..., 2^{k_j} - 1, \ j \in [d]\}$. Here and in what follows, we use the notations: $\ZZ_+:= \{s \in \ZZ: s \ge 0 \}$; $xy := (x_1y_1,..., x_dy_d)$; 
$2^x := (2^{x_1},...,2^{x_d})$;
$|x|_1 := \sum_{i=1}^d |x_i|$ for $x, y \in {\RR}^d$;
$[n]$ denotes the set of all natural numbers from $1$ to $n$; $x_i$ denotes the $i$th coordinate 
of $x \in \RR^d$, i.e., $x := (x_1,..., x_d)$. 

The sparse grids $G^d(m)$ for sampling recovery and numerical integration were first considered by Smolyak 
\cite{Sm63}. In \cite{Te85}--\cite{Te93b} and  
\cite{Di90}--\cite{Di92} Smolyak's construction was developed for studying the sampling recovery for periodic Sobolev classes  and Nikol'skii classes  having mixed smoothness. Recently, 
the sampling recovery  for periodic Besov classes  having mixed smoothness has been investigated in  
\cite{SU07,U08}.
For non-periodic functions of mixed smoothness linear sampling algorithms on Smolyak grids have been recently studied in  \cite{Tr10} $(d=2)$, \cite{Di11, SU11}, using the mixed tensor product B-splines. Methods of approximation by arbitrary linear combinations of translates of the Korobov kernel $\kappa_{r,d}$ on Smolyak grids of functions from the Korobov space $K^r_2(\TTd)$ which is a reproducing kernel Hilbert space with the associated kernel 
$\kappa_{r,d}$,  have constructed in \cite{DM13}. 

In numerical applications for approximation problems involving a large number of variables, Smolyak grids was first considered in
\cite{Ze91} in parallel algorithms for numerical solving PDEs. 
Numerical integration on Smolyak grids was investigated in \cite{GeGr98}. For non-periodic functions of mixed smoothness of integer order, linear sampling algorithms on Smolyak grids have been investigated in \cite{BG04} employing hierarchical Lagrangian polynomials multilevel basis. There is a very large number of papers on Smolyak grids and their modifications in various problems of approximations, sampling recovery and integration with applications in data mining, mathematical finance, learning theory, numerical solving of PDE and stochastic PDE, etc. to mention all of them. The reader can see the surveys in \cite{BG04, GN09, GeGr08} and the references therein.  For recent further developments and results see in \cite{GHo10, GH13a, GH13b, GaHe09, BGGK13}.

The Smolyak grids $G^d(m)$ and their various modifications 
(in partucular, the grids $\mathring{G}(m)$ and  $G^\nu(m)$) are very sparse. 
The number of knots of $G^d(m)$ is smaller than $\frac{1}{(d-1)!}2^m m^{d-1}$ and is much smaller than $2^{dm}$, the number of knots of  corresponding standard full grids. However, for periodic functions having mixed smoothness,  they give the same error of the sampling recovery on the standard full grids.
 
Quasi-interpolation based on scaled B-splines with integer knots, possesses good local and approximation properties for smooth functions, see \cite[p. 63--65]{BHR93}, \cite[p. 100-107]{C92}. It can be an efficient tool in some high-dimensional approximation problems, especially in applications ones. Thus, one of the important bases for sparse grid high-dimensional approximations having various applications, are the Faber functions (hat functions) which are piecewise linear B-splines of second order  
\cite{BG04,GeGr08, GHo10, GH13a, GH13b, GaHe09, BGGK13}.  The representation by Faber basis can be obtained by the B-spline quasi-interpolation (see, e. g., \cite{Di11}).                                              

The object of our interest in sampling recovery are functions having Lipschitz-H\"older mixed smoothness $\alpha > 0$. Let us introduce the space $\Ha$ of all such functions. 
For univariate functions $f$ on $\RR$, the $r$th difference operator $\Delta_h^r$ is defined by 
\begin{equation*}
\Delta_h^r(f,x) := \
\sum_{j =0}^r (-1)^{r - j} \binom{r}{j} f(x + jh).
\end{equation*}
If $u$ is any subset of $[d]$, for multivariate functions on $\RRd$
the mixed $(r,u)$th difference operator $\Delta_h^{r,u}$ is defined by 
\begin{equation*}
\Delta_h^{r,u} := \
\prod_{i \in u} \Delta_{h_i}^r, \ \Delta_h^{r,\varnothing} = I,
\end{equation*}
where the univariate operator
$\Delta_{h_i}^r$ is applied to the univariate function $f$ by considering $f$ as a 
function of  variable $x_i$ with the other variables held fixed. 
If $0 < \alpha \le r$, 
we introduce the semi-norm 
$|f|_{\Hau}$ for functions $f \in C(\TTd)$ by
\begin{equation} \nonumber
|f|_{\Hau}:= \
 \sup_{h \in \IId} \ \prod_{i \in u} h_i^{-\alpha}\|\Delta_h^{r,u}(f)\|_{C(\TTd)}
\end{equation}
(in particular, $|f|_{\Ha(\varnothing)} = \|f\|_{C(\TTd)}$).
The Lipschitz-H\"older space 
$\Ha$ of mixed smoothness $\alpha$ is defined as the set of  functions $f \in C(\TTd)$ 
for which the  norm 
\begin{equation*} 
\|f\|_{\Ha}
:= \ 
 \sup_{u \subset [d]} |f|_{\Hau}
\end{equation*}
is finite. The non-periodic space $H^\alpha_\infty(\IId)$ can be defined in a similar way with a slight modification.

In what follows, for the space $\Ha$, the parameters $\alpha$ and $p$ are fixed and therefore their value is considered as a constant. Denote by $\Ua$ the unit ball in $\Ha$. Paralleling with $\Ua$ we will consider the sampling recovery of functions from its subsets $\Uas$ and $\Uan$, $1 \le \nu \le d$. The set  $\Uas$ consists of all functions $f \in \Ua$ such that $f(x) =0$ if $x_j=0$  for some index $j \in [d]$.  It can be seen also as the subset of the unit ball in the space $H^\alpha_\infty(\IId)$ with zero boundary condition.
The set  $\Uan$ consists of all functions $f \in \Ua$ having at most $\nu$ active variables. 
Here, we say that the $j$th variable $x_j$ is  active for a function $f$ on $\TTd$ if for all $i \not=j$ there are $x_i^*$ such that the univariate function $g(t) = f(x_1^*,..., x_{j-1}^*,t,x_{j+1}^*,...,x_d^*)$ is not a constant as a univariate function in $t$. With this definition we have $U^{\alpha,d}_\infty = \Ua$. The set  $\Uan$ can be a model of the objects in a $d$-variate space depending only a few $\nu$ (much smaller than $d$) of variables without in general, exact information about them.

For sampling recovery of functions $f \in \Uas$, we use the linear sampling algorithm
\begin{equation} \nonumber
\mathring{S}_m(\mathring{\Phi}_m,f)
:= \
\sum_{\xi \in \mathring{G}(m)} f(\xi) \varphi_\xi,
\end{equation}
on the Smolyak grids 
\begin{equation}  \nonumber
\mathring{G}(m)
:= \
\{\xi = 2^{-k} s: k \in \NNd, \ |k|_1 = m, \ s \in \mathring{I}^d(k)\}
\end{equation}
where $\mathring{\Phi}_m = \{\varphi_\xi\}_{\xi \in \mathring{G}(m)}$ is a family of functions on $\TTd$,
$\mathring{I}^d(k) := \{s \in \ZZdp: s_j = 1,..., 2^{k_j} - 1, \ j \in [d]\}$.
For sampling recovery of functions  $f \in \Uan$, we use the linear sampling algorithm
\begin{equation} \nonumber
S^\nu_m(\Phi_m^\nu,f)
:= \
\sum_{\xi \in G^\nu(m)} f(\xi) \varphi_\xi,
\end{equation}
on the  Smolyak grids
\begin{equation} \nonumber 
G^\nu(m)
:= \
\{\xi = 2^{-k} s: k \in \NNd, \ |k|_1 = m, |\supp(k)| \le \nu, \ s \in I^d(k)\},
\end{equation}
where $\Phi_m^\nu = \{\varphi_\xi\}_{\xi \in G^\nu(m)}$ is a family of functions on $\TTd$,  
$I^d(k):= \{s \in \ZZdp: s_j = 0,1,..., 2^{k_j} - 1, \ j \in [d]\}$, $\operatorname{supp}(k)$ denotes the support of $k$, i.e., the subset of all $j \in [d]$ such that $k_j \not= 0$, and $|A|$ denotes the cardinality of the finite set $A$.

Let us notice the following.  The number of points in the grid $\mathring{G}(m)$ is 
\begin{equation}  \label{[G^d_circ(m)]}
|\mathring{G}(m)|
\ =  \
\sum_{k \in \NNd: \ |k|_1 = m} \ \prod_{j=1}^d (2^{k_j} - 1)
\ \le \
2^m \binom{m-1}{d-1} 
\ = \
|G^d(m)|
, \ m \ge d,
\end{equation}
and the number $G^\nu(m)$ of points in the grid $G^\nu(m)$ is 
\begin{equation}  \label{[G^nu(m)]}
|G^\nu(m)|
\ =  \
\sum_{|u| \le \nu} \ \sum_{\supp(k)= u, \, |k|_1 = m} \ \prod_{j \in \supp(k)} (2^{k_j} - 1)
\ \le \
2^m \binom{d}{\nu}\binom{m-1}{\nu-1}, \ m \ge \nu.
\end{equation}
The choice of the Smolyak grid $G^\nu(m)$ for sampling recovery of functions $f \in \Uan$ is quite natural since in general, their active variables are unknown.
Moreover, the grids $G^\nu(m)$ polynomially depending on $d$, is full if and only if $m \ge \nu$, and the grids 
$\mathring{G}(m)$ are empty if $m < d$. For this reason, we are interested in constructing linear sampling algorithms on the grids $G^\nu(m)$ and $\mathring{G}(m)$ for $m \ge \nu$ and $m \ge d$, respectively. 

In high-dimensional approximation applications, knowing the number of sampled function values on the Smolyak grids $G^\nu(m)$ and $\mathring{G}(m)$ as in \eqref{[G^d_circ(m)]} and \eqref{[G^nu(m)]} and linear sampling algorithms
$\mathring{S}_m(\mathring{\Phi}_m,f)$
 and $S^\nu_m(\Phi^\nu_m,f)$, one attempts to estimate the errors of $\|f - \mathring{S}_m(\mathring{\Phi}_m,f)\|_p$
 and $\|f - S^\nu_m(\Phi^\nu_m,f)\|_p$ as a function of three variables $m$, $d$, $\nu$ when $m$ and $d$ may be very large. To study the optimality of these algorithms let us introduce the quantity of optimal sampling recovery $\mathring{s}_m(F_d)_p$ on Smolyak grids 
$\mathring{G}(m)$ with respect to the function class $F_d$ by
\begin{equation} \nonumber
\mathring{s}_m(F_d)_p
\ := \ \inf_{\mathring{\Phi}_m} \  \sup_{f \in F_d} \, \|f - \mathring{S}_m(\mathring{\Phi}_m,f)\|_p, 
\end{equation} 
and 
the quantity of optimal sampling recovery $s^\nu_m(F_d)_p$ on Smolyak grids 
$G^\nu(m)$ with respect to the function class $F_d$ by
\begin{equation} \nonumber
s^\nu_m(F_d)_p
\ := \ \inf_{\Phi^\nu_m} \  \sup_{f \in F_d} \, \|f - S^\nu_m(\Phi^\nu_m,f)\|_p.
\end{equation} 

In traditional estimation of the error, a typical form of  lower and upper bounds of $s^d_m(\Ua)_p$ and $\mathring{s}_m(\Uas)_p$ is 
\begin{equation} \label{[TradEstimation]}
c'(d)\, 2^{-\alpha m} m^{d-1}
\ \le \
\mathring{s}_m(\Uas)_p
\ \le \
s^d_m(\Ua)_p
\ \le \
c(d)\, 2^{-\alpha m} m^{d-1},
\end{equation}
where $c(d)$ and $c'(d)$ are a constant not explicitly computed in $d$ 
(see, e.g., \cite{Di11, DU14, Te93b}). 
In high-dimensional approximation problems using function values information, the parameter $m$ is main in the study of convergence rates of the approximation error with respect to $m$ going to infinity. However, the dimension $d$ may hardly affect this rate when $d$ is large. 
In the present paper, we establish upper and lower bounds for $\mathring{s}_m(\Uas)_p$ and $s^\nu_m(\Uan)_p$ explicitly in $m$ and $\nu,d$ as a function of three variables 
$m$, $d$, $\nu$. Notice for example, that in traditional estimates the term 
$2^{-\alpha m} m^{d-1}$ is {\em a priori} split from constants $c(d, \alpha,p)$ which are actually a function of dimension parameter $d$, and therefore, any high-dimensional estimate based on them may lead to a rough bound. By a combinatoric argument, it is more natural and suitable to investigate  upper and lower bounds for $\mathring{s}_m(\Uas)_p$ and $s^\nu_m(\Uan)_p$ in the form 
\begin{equation} \label{[<ring{s}_m<]}
\mathring{C}'(d)\, 2^{-\alpha m} \binom{m}{d-1}
 \le   
\mathring{s}_m(\Uas)_p
\le 
\mathring{C}(d)\, 2^{-\alpha m} \binom{m}{d-1}, 
\end{equation}
\begin{equation} \label{[<s^nu_m<]}
C'(\nu)\,  2^{-\alpha m} \binom{m}{\nu-1}
 \le   
s^\nu_m(\Uan)_p
\ \le \  
C(\nu)\,  2^{-\alpha m} \binom{m}{\nu-1},
\end{equation}
or in the more refined form
\begin{equation} \label{[<ring{s}_m<](2)}
\mathring{A}'(m,d)\, 2^{-\alpha m} 
 \le   
\mathring{s}_m(\Uas)_p
\le 
\mathring{A}(m,d)\, 2^{-\alpha m}, 
\end{equation}
\begin{equation} \label{[<s^nu_m<](2)}
A'(m,\nu)\,  2^{-\alpha m}
 \le   
s^\nu_m(\Uan)_p
\ \le \  
A(m,\nu)\,  2^{-\alpha m}. 
\end{equation}

In the present paper,  we estimate $\mathring{s}_m(\Uas)_p$ and $s^\nu_m(\Uan)_p$ 
in the forms \eqref{[TradEstimation]}, \eqref{[<ring{s}_m<]}--\eqref{[<s^nu_m<]} and  
\eqref{[<ring{s}_m<](2)}--\eqref{[<s^nu_m<](2)},
stressing in finding lower and upper bounds for the accompanying "constants"
 explicit in $d$ or in $\nu$ or in $m,\nu$, respectively.
To obtain these upper we construct linear sampling algorithms on the Smolyak grids $\mathring{G}(m)$ and  $G^\nu(m)$ based on a B-spline quasi-interpolation series especially, on Faber series related to the well-known hat functions. Due to the complicatedness, we are restricted to give these lower bounds only for the case $\alpha \le 2$.  As consequences we obtain upper and lower bounds for the quantities of optimal cubatute formula $\mathring{\Int}_m(\Uas)_p$ and  $\Int^\nu_m(\Uan)_p$  explicit in $m$ and $d, \nu$.

Related to the problems investigated in the present paper, is the problem of hyperbolic cross approximation of functions having mixed smoothness in high-dimensional setting in terms of $n$-widths and 
$\varepsilon$-dimensions which have been investigated in \cite{CD13, DU13, KSU14}.

The paper is organized as follows. In Section \ref{Faber}, we establish upper and lower bounds and construct linear sampling algorithms on Smolyak grids based on  Faber series for $\mathring{s}_m(\Uas)_p$ and 
$s^\nu_m(\Uan)_p$ for $0 < \alpha \le 2$. As consequences, we derive upper and lower bounds for the error of cubature formulas on Smolyak grids and of optimal integration on Smolyak grids, and  $\mathring{\Int}_m(\Uas)_p$ and 
$\Int_m^\nu(\Uan)_p$.  In Section \ref{Sampling recovery[alpha>2]}, we extend the results on upper bounds of $s^\nu_m(\Uan)_p$ and $\Int_m^\nu(\Uan)_p$ obtained in Section~\ref{Faber} to the case of arbitrary mixed smoothness 
$\alpha > 0$, based on high-order B-spine quasi-interpolation representations for function from $\Ha$.  In Section~\ref{Examples}, we give some example of polynomials inducing  quasi-interpolation operators.

\section{Sampling recovery based on Faber series}
\label{Faber} 

\subsection{Faber series}

Let  $M_2(x)\ = \ (1 - |x-1|)_+$, $x \in \II$, be the piece-wise linear B-spine with knot at $0,1,2$ (the hat function), where 
$x_+:= \max(x,0)$ for $x \in \RR$. Since the support of functions $M_2(2^{k+1}\cdot)$ for $k \in \ZZ_+$ is the interval $[0,2^{-k}] \subset \II$, we can extend these functions to an $1$-periodic function on the whole $\RR$. Denote this periodic extension by $\varphi_k$.
The univariate hat functions $\varphi_{k,s}$ are defined by
\begin{equation*} 
\varphi_{0,0}:= 1, 
\quad
\varphi_{k,s}:= \varphi_k(\cdot - 2s), \ k > 0 \ s \in Z(k),
\end{equation*}
where $Z(0) := \{0\}$ and  $Z(k) := \{0,1,..., 2^{k-1} - 1\}$.
Put $Z^d(k):= \prod_{i=1}^d Z(k_i)$. 
For $k \in {\ZZ}^d_+$, $s \in Z^d(k)$, define the $d$-variate hat functions
\begin{equation*} 
\varphi_{k,s}(x)
\ := \
\prod_{i=1}^d \varphi_{k_i,s_i}(x_i),
\end{equation*}
and the $d$-variate periodic Faber system $\Ff_d$  by 
\begin{equation*} 
\Ff_d := \{\varphi_{k,s}: s\in Z^d(k),\ k \in \ZZdp\}. 
\end{equation*}

For functions $f$ on $\TT$, we define the univariate linear functionals $\lambda_{k,s}$ by 
\begin{equation*} 
\lambda_{k,s}(f) \ := \
- \frac {1}{2} \Delta_{2^{-k}}^2 (f,2^{-k + 1}s),   \, k > 0, \ 
\text{and} \ \lambda_{0,0}(f) \ := \ f(0). 
\end{equation*}
Let the $d$-variate linear functionals
$\lambda_{k,s}$ be defined as
\begin{equation*}   
\lambda_{k,s}(f) 
\ := \   
\lambda_{k_1,s_1}(\lambda_{k_2,s_2}(... \lambda_{k_d,s_d}(f))),
\end{equation*}
where the univariate functional
$\lambda_{k_i,s_i}$ is applied to the univariate function $f$ by considering $f$ as a 
function of  variable $x_i$ with the other variables held fixed. 

\begin{lemma} \label{lemma[convergence(d)]}
 The $d$-variate periodic Faber system $\Ff_d$ is a basis in $C(\TTd)$. Moreover, a function 
$f \in C(\TTd)$ can be represented by the Faber series 
\begin{equation} \label{eq[FaberRepresentation]}
f
\ = \
\sum_{k \in \ZZdp} q_k(f) 
\ = \
\sum_{k \in \ZZdp} \sum_{s \in Z^d(k)} \lambda_{k,s}(f)\varphi_{k,s}, 
\end{equation}
converging in the norm of $C(\TTd)$.
\end{lemma}
\begin{proof}
For the univariate case ($d=1$), this lemma can be deduced from its well-known counterpart for non-periodic functions on $\II$ (see, e.g., \cite[Theorem 1, Chapter VI]{KS89}). For the multivariate case ($d>1$), it can be proven by the tensor product argument.
\end{proof}

Put $\ZZdpu:= \{k \in \ZZdp: \operatorname{supp} (k) = u\}$ for $u \subset [d]$, and use the convention $x^0 = 1$ for $x \in [0, \infty]$. 
\begin{theorem} \label{theorem[FaberRepresentation]}
Let $0 < p \le \infty$ and $0 < \alpha \le 2$. Then a function 
$f \in \Ha$ can be represented by the series \eqref{eq[FaberRepresentation]}
converging in the norm of $C(\TTd)$.
Moreover, we have for every $k \in \ZZdpu$,
\begin{equation} \label{ineq[|q_k(f)|_p]}
\|q_k(f) \|_p
  \ \le \ 
2^{-|u|} \, (p+1)^{- |u|/p} \, 2^{- \alpha|k|_1}.
\end{equation}  
\end{theorem}

\begin{proof} The first part of the lemma on representation and convergence is in Lemma \ref{lemma[convergence(d)]}. Let us prove \eqref{ineq[|q_k(f)|_p]}. We first consider the case $p = \infty$.
Since $0 \le  \sum_{s \in Z^d(k)}  \varphi_{k,s} (x) \le 1 $, we have for every $k \in \ZZdpu$ and 
$x \in \TTd$,
\begin{equation*} 
\begin{aligned}
|q_k(f)(x)|
 \ &\le \ 
 \sup_{s \in Z^d(k)} \, |\lambda_{k,s}(f)|
 \sum_{s \in Z^d(k)}  \varphi_{k,s} (x)
 \ \le \ 
\sup_{s \in Z^d(k)} \, |\lambda_{k,s}(f)| \\[1.5ex]
 \ &\le \ 
\sup_{s \in Z^d(k)} \,\Big|\prod_{j \in u} \big[-\frac{1}{2} \Delta_{2^{-k_j}}^2 f(x^{k,s}(u))\big]\Big| 
 \ \le \ 
2^{- |u|}\, \omega_r^u(f, 2^{-k}) 
 \ \le \ 
2^{- |u|} \, 2^{- \alpha|k|_1} \, |f|_{\Hau}.
\end{aligned}
\end{equation*} 
This proves \eqref{ineq[|q_k(f)|_p]} for $p= \infty$.  Hence, if $1 \le p < \infty$, 
we have also that for every $k \in \ZZdpu$,
\begin{equation*} 
\begin{aligned}
\|q_k(f) \|_p^p
 \ &= \ 
 \int_{\TTd}\Big| \sum_{s \in Z^d(k)} \, \lambda_{k,s}(f) \varphi_{k,s} (x) \Big|^p \, dx  
  \ \le \ 
 \sup_{s \in Z^d(k)} \, |\lambda_{k,s}(f)|^p \, 
\int_{\TTd}\Big| \sum_{s \in Z^d(k)} \varphi_{k,s}(x) \Big|^p \, dx  \\[1.5ex]
 \ &= \ 
 \sup_{s \in Z^d(k)} \, |\lambda_{k,s}(f)|^p \, 
 |Z^d(k)| \int_{\TTd}|\varphi_{k,0}(x)|^p \, dx  \\[1.5ex]
\ &= \ 
 \sup_{s \in Z^d(k)} \, |\lambda_{k,s}(f)|^p \, 
 |Z^d(k)| \prod_{j \in u}2 \int_{0}^{2^{-k_j}}|2^{k_j}x_j|^p \, dx_j  \\[1.5ex]
 \ &\le \ 
 \big[2^{- |u|} \, 2^{- \alpha|k|_1} \, |f|_{\Hau}\big]^p \, 
 \big[2^{-|u|} 2^{|k|_1}\big] \big[2^{|u|} (p+1)^{-|u|} 2^{-|k|_1}\big] \\[1.5ex]
 \ &= \ 
2^{-|u|} \, (p+1)^{- |u|} \, 2^{- p\alpha|k|_1} \, |f|_{\Hau}^p.
\end{aligned}
\end{equation*}
This proves \eqref{ineq[|q_k(f)|_p]} for $1 \le p < \infty$. Finally, the case  $0< p < 1$ can be proven in a similar way starting from the inequality
\begin{equation*} 
\|q_k(f) \|_p^p
 \ \le \ 
 \sum_{s \in Z^d(k)} \, |\lambda_{k,s}(f)|^p \int_{\TTd}| \varphi_{k,s} (x)|^p \, dx.
 \end{equation*}
The proof is complete. 
\end{proof}

\subsection{Auxiliary lemmas}

For  $m,n \in \NN$ with $m \ge n$, we introduce the function $F_{m,n}: (0,1) \to \RR$ by 
\begin{equation}  \nonumber
F_{m,n}(t)
:= \ 
\sum_{s=0}^\infty \binom{m+s}{n} t^s.
\end{equation}
From the definition it follows that $F_{m,n}(t)= \binom{m}{n} + tF_{m+1,n}(t)$ and consequently,
\begin{equation}  \label{[<F_{m,n}<]}
tF_{m+1,n}(t)
\ < \ 
F_{m,n}(t)
\ < \
F_{m+1,n}(t), \ t \in (0,1).
\end{equation}

We will need following equation proven in \cite[(3.67), p.29]{BG04}.
\begin{equation}  \label{eq[BG]}
F_{m,n}(t)
\ = \
\frac{1}{1-t}\, \sum_{s=0}^n \binom{m}{s}\, \biggl(\frac{t}{1-t}\biggl)^{n-s}.
\end{equation}
For nonnegative integer $n$, we define the function 
\begin{equation} \nonumber
b_n(t)
:= \
\begin{cases}
(1-2t)^{-1}, \ & t<1/2,\\[1.5ex]
2(n+1), \ & t=1/2,\\[1.5ex]
(2t-1)^{-1}[t/(1-t)]^{n+1}, \ & t>1/2.
\end{cases}
\end{equation}

\begin{lemma} \label{lemma[sum(1)]}
Let $m,n \in \NN$, $m \ge 2n$, and $0 < t <1$. Then we have
\begin{equation}  \label{eq[sum(1)]}
F_{m,n}(t)
\ \le  \
\binom{m}{n}\, b_n(t).
\end{equation}
\end{lemma}

\begin{proof}
We have
 \begin{equation}  \nonumber
\sum_{s=0}^n \binom{m}{s} 
\ \le \
(n+1) \, \binom{m}{n}.
\end{equation} 
Hence, by  \eqref{eq[BG]} the case $t=1/2$ of the inequality  \eqref{eq[sum(1)]} is proven. 
Next, we have for $x>0$ and $x\not=1$,
\begin{equation}  \nonumber
\sum_{s=0}^n \binom{m}{s} x^s
\ \le \
\binom{m}{n}\,\sum_{s=0}^n x^s
\ = \
\binom{m}{n}\,\frac{x^{n+1}-1}{x-1},
\end{equation}
and
\begin{equation} \nonumber
\frac{x^{n+1}-1}{x-1}
\ \le \
\begin{cases}
(x-1)^{-1} x^{n+1}, \ & x>1,\\[1.5ex]
(1-x)^{-1}, \ & x<1.
\end{cases}
\end{equation}
By using the last two inequalities for $x = (1-t)/t$ from \eqref{eq[BG]} we prove the cases $t<1/2$ and $t>1/2$ of the inequality  \eqref{eq[sum(1)]}. 
\end{proof}

 \begin{lemma} \label{lemma[sum(2)]}
Let $m,n \in \NN$, $m \ge n$, and $0 < t <1$. Then we have
\begin{equation}  \nonumber
F_{m,n}(t)
\ \le  \
\frac{1}{1-t}\, \biggl(\frac{t}{1-t}\biggl)^n\, \exp\biggl(\frac{1-t}{t}\biggl)\, m^n
\end{equation}
\end{lemma}

\begin{proof}
We have for $x>0$,
\begin{equation}  \nonumber
\sum_{s=0}^n \binom{m}{s} \, x^s
\ = \
\sum_{s=0}^n \frac{m!}{s! (m-s)!} \, x^s
\ \le \
\sum_{s=0}^n m^s \, \frac{x^s}{s!}
\ \le \
m^n \, e^x.
\end{equation}
Using this estimate from \eqref{eq[BG]} we deduce the lemma. 
\end{proof}



\subsection{Upper bounds}
\label{Upper bounds, alpha<2}

In this subsection, we employ the representation by the Faber series \eqref{eq[FaberRepresentation]}
to construct linear sampling algorithms $\mathring{S}_m(\mathring{\Phi}_m,\cdot)$ and $S^\nu_m(\Phi_m^\nu,\cdot)$
for functions from $\Uas$ and $\Uan$, respectively. By use of Theorem~\ref{theorem[FaberRepresentation]}
and auxiliary lemmas in the previous subsection we establish upper bounds for the error of the sampling recovery by these algorithms and therefore, for $\mathring{s}_m(\Uas)_p$ and $s^\nu_m(\Uan)_p$.

We first construct linear sampling algorithms functions from $\Uas$. 
By the definition we can see that
\begin{equation*}
\Uas
\ = \ 
\biggl\{f \in \Ua: f \ = \ \sum_{k \in \NNd} q_k(f) \biggl\}.
\end{equation*}
For $m \in {\ZZ}_+$, we introduce the operator $\mathring{R}_m$ for $f \in \Uas$ by 
\begin{equation*}
\mathring{R}_m(f) 
:= \ 
\sum_{k \in \NNd: \, |k|_1 \le m} q_k(f)
\ = \
\sum_{k \in \NNd: \, |k|_1 \le m} \ \sum_{s \in Z^d(k)} \lambda_{k,s}(f)\varphi_{k,s}.
\end{equation*} 
The operator $\mathring{R}_m$ defines a linear sampling algorithm $\mathring{S}_m(\mathring{\Psi}_m,f)$
 on the grid $\mathring{G}(m)$ by 
\begin{equation*} 
\mathring{R}_m(f) 
\ = \ 
\mathring{S}_m(\mathring{\Psi}_m,f) 
\ = \ 
\sum_{\xi \in \mathring{G}(m)} f(\xi) \psi_{\xi},
\end{equation*} 
where $\mathring{\Psi}_m:= \{\psi_{\xi}\}_{\xi \in \mathring{G}(m)}$,
\begin{equation*} 
\psi_\xi(x)
\ = \
\prod_{i=1}^d \psi_{ 2^{-k_i} s_i}(x_i), \ \xi = 2^{-k}s, \ k \in \NNd, \ s \in I^d(k), 
\end{equation*}
and the univariate functions $\psi_{k,s}$, $k \in \NN, \ s \in I^1(k)$, are defined by
\begin{equation*} 
\psi_{ 2^{-k} s}
\ =  \ 
\begin{cases}
1- \varphi_{0,0}, & \ k=1, \ s=0 \\[1.5ex]
\varphi_{0,0}, & \ k=1, \ s=1 \\[1.5ex]
- \frac{1}{2} ( \varphi_{k,j} + \varphi_{k,j-1}), & \ k>1, \ s=2j,  \ 0\le j<2^{k-1}-1, \\[1.5ex]
\varphi_{k,j}, & \ k>1, \ s = 2j+1, \ 0<j<2^{k-1}-1.
\end{cases}
\end{equation*}

For $0 < p \le \infty$, $\alpha >0$ and $m \ge l - 1$, put  
\[
b \ = \ b(\alpha,p):= \ 2(2^\alpha - 1)(p+1)^{1/p}, 
\]
\begin{equation}  \nonumber
\beta(l,m)
\ = \
\beta(\alpha,l,m)
:= \
(2^\alpha - 1)^{-l}\, \sum_{s=0}^{l-1} \binom{m}{s}\, (2^\alpha - 1)^s, \quad 
\beta(0,m)
:= \ 1,
\end{equation}
and 
\begin{equation} \nonumber
a^\circ(d)
\ = \
a^\circ(\alpha, p, d)
:= \ |2^\alpha -2|^{- |\operatorname{sgn}(\alpha - 1)|}\, [2(p+1)^{1/p}]^{-d} \times
\begin{cases}
1, \ & \alpha > 1,\\[1.5ex]
d, \ & \alpha = 1,\\[1.5ex]
(2^\alpha -1)^{-d}, \ & \alpha < 1.
\end{cases}
\end{equation}

The following theorem gives  upper bounds for the error $\|f - \mathring{R}_m(f)\|_p$ for $f \in \Uas$ and therefore, for $\mathring{s}_m(\Uas)_p$.

\begin{theorem} \label{theorem[|f - R_m(f)|_p,Has]}
Let  $0 < p \le \infty$ and $0 < \alpha \le 2$. Then we have for every
$m \ge d$, 
\begin{equation}  \label{ineq[|f - R_m(f)|_p,Has(1)]}
\begin{aligned}
\mathring{s}_m(\Uas)_p
\ \le \ 
\sup_{f \in \Uas} \|f - \mathring{R}_m(f)\|_p
\ &\le \  
 [2(p+1)^{1/p}]^{-d}\, \beta(d,m)\, 2^{- \alpha m} \\[1ex]
\ &\le \  
\exp(2^\alpha - 1) \, b^{-d} \,
2^{- \alpha m} \, m^{d-1}. 
\end{aligned}
\end{equation}
Moreover, if in addition, $m \ge 2(d-1)$,
\begin{equation}  \label{ineq[|f - R_m(f)|_p,Has(2)]}
\mathring{s}_m(\Uas)_p
\ \le \ 
\sup_{f \in \Uas} \|f - \mathring{R}_m(f)\|_p
\ \le \  
 a^\circ(d)\, 2^{- \alpha m} \, \binom{m}{d-1}. 
\end{equation}
\end{theorem}

\begin{proof}
Let us prove the lemma for $1 \le p \le \infty$. It can be proven in a similar way with a slight modification for $0 < p <1$.  From Theorem \ref{theorem[FaberRepresentation]},  \eqref{eq[BG]} and Lemma \ref{lemma[sum(2)]} it follows that for every  $f \in \Uas$ and and every $m \ge d$,  
\begin{equation} \label{[|f - R_m(f)|_p<]}
\begin{aligned}
\|f - \mathring{R}_m(f)\|_p
 &\le  
\sum_{k \in \NNd: \, |k|_1 > m}\|q_k(f) \|_p 
 \le  
\sum_{k \in \NNd: \, |k|_1 > m}
2^{-d} \, (p+1)^{- d/p} \, 2^{- \alpha|k|_1}  \\[1.5ex]
 &=  
2^{-d} \, (p+1)^{- d/p} \sum_{k \in \NN^d: \, |k|_1 > m}
  2^{- \alpha|k|_1}  
 =  
 2^{-d} \, (p+1)^{- d/p} \sum_{j=1}^\infty
 \binom{m+j-1}{d-1}\, 2^{- \alpha (m+j)}  \\[1.5ex]
 &=  
 2^{- \alpha (m+1)}\, 2^{-d} \, (p+1)^{- d/p} 
F_{m,d-1}(2^{- \alpha})
 = 
 [2(p+1)^{1/p}]^{-d}\, 2^{- \alpha m} \beta(d,m)  \\[1.5ex]
 &\le
\exp(2^\alpha - 1) \, b^{-d} \,
2^{- \alpha m} \, m^{d-1}. 
\end{aligned}
\end{equation}
The inequalities  \eqref{ineq[|f - R_m(f)|_p,Has(1)]} are proven. The inequality \eqref{ineq[|f - R_m(f)|_p,Has(2)]} can be derived from \eqref{[|f - R_m(f)|_p<]} by applying  Lemma \ref{lemma[sum(1)]} for $t=2^{- \alpha}$.
\end{proof}

Notice that some upper bounds of $\|f - \mathring{R}_m(f)\|_p$ with $p = 2, \infty$ for functions with zero boundary condition from the Sobolev space $H^2(\IId)$ were established in \cite{BG04}.

We next construct linear sampling algorithms functions from $\Uan$. The definition of $\Uan$ implies that 
\begin{equation*}
\Uan
\ = \ 
\biggl\{f \in \Ua: \ \exists u \subset [d], \ |u| = \nu: \ f \ = \ \sum_{v \subset u}\,\sum_{k \in \ZZdp(v)} q_k(f) \biggl\}.
\end{equation*}
For $m \in {\ZZ}_+$, we define the operator $R^\nu_m$ by 
\begin{equation*}
R^\nu_m(f) 
:= \ 
\sum_{k \in \ZZdp: \, |\supp(k)| \le \nu, \ |k|_1 \le m} q_k(f)
\ = \
\sum_{k \in \ZZdp: \, |\supp(k)| \le \nu, \ |k|_1 \le m} \ \sum_{s \in Z^d(k)} \lambda_{k,s}(f)\varphi_{k,s}.
\end{equation*} 
For functions $f$ on $\TTd$, $R^\nu_m$ defines a linear sampling algorithm $S^\nu_m(\Psi^\nu_m,f)$
 on the Smolyak grid $G^\nu(m)$ by 
\begin{equation*} 
R^\nu_m(f) 
\ = \ 
S^\nu_m(\Psi^\nu_m,f) 
\ = \ 
\sum_{\xi \in G^\nu(m)} f(\xi) \psi_{\xi}, 
\end{equation*} 
where $\Psi^\nu_m:= \{\psi_{\xi}\}_{\xi \in G^\nu(m)}$. Notice that the construction of the operator $R^\nu_m$ is quite reasonable since in our setting we assume that the active variables of a $f \in \Uan$ are unknown.

Put
\[
\gamma(\nu,m)
\ = \
\gamma(\alpha,p,\nu,m)
:= \
 \sum_{l=0}^\nu \binom{\nu}{l}\, 2^{-l} \, (p+1)^{- l/p} \beta(\alpha,l,m),
\]
and  
\begin{equation} \nonumber
a(\alpha, p, \nu)
\ = \
a(\nu)
:= \ |2^\alpha -2|^{- |\operatorname{sgn}(\alpha - 1)|}\,[1 + 1/2(p+1)^{1/p}]^\nu \times
\begin{cases}
1, \ & \alpha > 1,\\[1.5ex]
\nu, \ & \alpha = 1,\\[1.5ex]
(2^\alpha -1)^{-\nu}, \ & \alpha < 1.
\end{cases}
\end{equation}

\begin{theorem} \label{theorem[|f - R^nu_m(f)|_p]}
Let $0 < p \le \infty$, $0 < \alpha \le 2$ and $1\le \nu \le d$. Then we have for every $m \ge \nu$, 
\begin{equation}  \label{ineq[|f - R^nu_m(f)|_p](1)}
\begin{aligned}
s^\nu_m(\Uan)_p
\ \le \ 
\sup_{f \in \Uan} \|f - R^\nu_m(f)\|_p
\ &\le \
  \gamma(\nu,m) \, 2^{- \alpha m}\\[1ex]
\ &\le \  \exp(2^\alpha - 1) \, (1+1/b)^\nu \, 2^{- \alpha m}\, m^{\nu-1}.
\end{aligned}
\end{equation}
Moreover,  if in addition, $m \ge 2(\nu-1)$,
\begin{equation}  \label{ineq[|f - R^nu_m(f)|_p](2)} 
s^\nu_m(\Uan)_p
\ \le \ 
\sup_{f \in \Uan} \|f - R^\nu_m(f)\|_p
\ \le \  
a(\nu)\,  2^{- \alpha m}\, \binom{m}{\nu-1}.
\end{equation}
\end{theorem}

\begin{proof} 
Let us prove the theorem for $1 \le p \le \infty$. It can be proven in a similar way with a slight modification for 
$0 < p <1$. Put $t=2^{- \alpha}$. We first consider the case $\nu = d$. 
From Theorem \ref{theorem[FaberRepresentation]}, \eqref{eq[BG]} and Lemma \ref{lemma[sum(2)]} it follows that for every  $f \in \Ua$ and and every $m \ge d$, 
\begin{equation} \label{estimation[|f - R_m(f)|_p](3)}
\begin{aligned}
\|f - R^d_m(f)\|_p
  \ &\le \ 
 \sum_{u \subset [d]}  \ \sum_{k \in \ZZdpu: \, |k|_1 > m}\|q_k(f) \|_p \\[1.5ex]
  \ &\le \ 
 \sum_{u \subset [d]} \ \sum_{k \in \ZZdpu: \, |k|_1 > m}
2^{-|u|} \, (p+1)^{- |u|/p} \, 2^{- \alpha|k|_1}  \\[1.5ex]
 \ &= \ 
 \sum_{l=0}^d \binom{d}{l}\, 2^{-l} \, (p+1)^{- l/p} \sum_{k \in \NN^l: \, |k|_1 > m}
  2^{- \alpha|k|_1}  \\[1.5ex]
 \ &= \ 
 \sum_{l=0}^d \binom{d}{l}\, 2^{-l} \, (p+1)^{- l/p} \sum_{j=1}^\infty
 \binom{m+j-1}{l-1}\, 2^{- \alpha (m+j)}  \\[1.5ex]
\ &= \ 
 2^{- \alpha (m+1)}\sum_{l=0}^d \binom{d}{l}\, 2^{-l} \, (p+1)^{- l/p}
F_{m,l-1}(t) \ = \ 2^{- \alpha m} \gamma(\nu,m).
\end{aligned}
\end{equation}
Further, applying Lemma \ref{lemma[sum(2)]} to $F_{m,l-1}(t)$ in \eqref{estimation[|f - R_m(f)|_p](3)} we get
\begin{equation} \nonumber
\begin{aligned}
\|f - R^d_m(f)\|_p
  \ &\le \ 
 2^{- \alpha (m+1)}\sum_{l=0}^d \binom{d}{l}\, 2^{-l} \, (p+1)^{- l/p}
\frac{1}{1-t}\, \biggl(\frac{t}{1-t}\biggl)^{l-1}\, \exp\biggl(\frac{1-t}{t}\biggl)\, m^{l-1} \\[1.5ex]
  \ &\le \ 
 2^{- \alpha (m+1)}\exp\biggl(\frac{1-t}{t}\biggl)\,t^{-1} m^{d-1}\sum_{l=0}^d \binom{d}{l}\, 2^{-l} \, (p+1)^{- l/p}
\, \biggl(\frac{t}{1-t}\biggl)^l \\[1.5ex]
\ &= \ 
 2^{- \alpha m}\exp\biggl(2^\alpha -1\biggl) m^{d-1}\sum_{l=0}^d \binom{d}{l}\, 2^{-l} \, (p+1)^{- l/p}
\, \biggl(2^\alpha -1\biggl)^{-l} \\[1.5ex]
\ &= \ 
\exp(2^\alpha -1) \, (1+1/b)^d \, 2^{- \alpha m}\, m^{d-1}.
\end{aligned}
\end{equation}
The inequality \eqref{ineq[|f - R^nu_m(f)|_p](1)} is proven.

Let us verify \eqref{ineq[|f - R^nu_m(f)|_p](2)}. Applying Lemma \ref{lemma[sum(1)]} to 
$F_{m,l-1}(t)$ in \eqref{estimation[|f - R_m(f)|_p](3)} we have  for $m \ge 2(d-1)$, 
\begin{equation} \nonumber
\begin{aligned}
\|f - R^d_m(f)\|_p
  \ &\le \ 
2^{- \alpha (m+1)}\sum_{l=0}^d \binom{d}{l} \binom{m}{l-1}\, 
2^{-l} \, (p+1)^{- l/p} \, b_{l-1}(2^{-\alpha})  \\[1.5ex]
\ &\le \ 
 2^{- \alpha (m+1)} \binom{m}{d-1}\sum_{l=0}^d \binom{d}{l}\, 
2^{-l} \, (p+1)^{- l/p} \, b_{l-1}(2^{-\alpha})  \\[1.5ex]
\ &= \
a(d)\, 2^{- \alpha m}\, \binom{m}{d-1}.
\end{aligned}
\end{equation}
This completes the proof of the theorem for the case $\nu = d$.
 
We now consider the case $\nu < d$.  Let $f \in \Uan$. Without the generality we can assume that the active variables of $f$ are among $x_1,...,x_\nu$. Hence, we have  
\begin{equation*}  
f(x) 
\ = \
\sum_{u \subset [\nu]} \ \sum_{k \in \ZZdpu} q_k(f)(x), \quad 
R^\nu_m(f) 
 \ = \ 
\sum_{u \subset [\nu]} \ \sum_{k \in \ZZdpu: \ |k|_1 \le m} q_k(f)(x). 
\end{equation*}
Considering $f$ as a function in $U^\alpha_\infty(\TT^\nu)$ and $R^\nu_m$ as an operator in $C(\TT^\nu)$, and applying the proven case $d=\nu$ to $\big\|f - R^\nu_m(f)\big\|_p$, we prove Theorem~\ref{theorem[|f - R^nu_m(f)|_p]}. 
\end{proof}

\subsection{Lower bounds}
\label{Lower bounds}

\begin{theorem} \label{LowerBnd,Has]}
Let  $1\le p \le \infty$ and $0 < \alpha \le 2$. Then we have every $m \ge d$, 
\begin{equation}  \label{LowerBnd,Has(1)]}
\mathring{s}_m(\Uas)_p
\ \ge \  
 2^{-7d}\, \beta(d,m)\, 2^{-\alpha m}
\ \ge \  
 (2^\alpha - 1)^{-1} \, 2^{-7d}\,
2^{- \alpha m} \, \binom{m}{d-1}, 
\end{equation}
and therefore,
\begin{equation}   \label{LowerBnd,Has(1a)]}
\mathring{s}_m(\Uas)_p
\ \ge \ 
 (2^\alpha -1)^{-1}\, 2^{-7d} \, (d-1)^{-(d-1)}\, 2^{- \alpha m} \, m^{d-1}.
\end{equation}
\end{theorem}

\begin{proof}
The inequality \eqref{LowerBnd,Has(1a)]} follows from \eqref{LowerBnd,Has(1)]} and the inequality 
$\binom{m}{d-1} \ge \big(\frac{m}{d-1}\big)^{d-1}$. Let us prove \eqref{LowerBnd,Has(1)]}.
Let  $M_4$ be the cubic cardinal B-spline with support $[0,4]$ and 
knots at the points $0,1,2,3,4$. Since the support of functions $M_4(2^{k+2}\cdot)$ for $k \in \ZZ_+$ is the interval $[0,2^{-k}]$, we can extend these functions to an $1$-periodic function on the whole $\RR$. Denote this periodic extension by $\mu_k$.
Define the univariate nonnegative functions $g_{k,s}$ and $g_k$ on $\TT$ by
\begin{equation} \label{[g_{k,s}]}
g_{k,s}(x):= \ \mu_k(x - 4s), \ k \in \ZZ_+, \ s = 0,...,2^k-1;   
\quad
g_k(x):= \ \sum_{s=0}^{2^k-1} g_{k,s}(x).
\end{equation} 
From the identity $M_4^{(2)}(x) = M_2(x) - 2M_2(x - 1) + M_2(x - 2)$, $x \in \RR$, it follows that 
$
|M_4^{(2)}(x)| \le 2M_2(2^{-1}x), \ x \in \RR,
$
and consequently, 
\begin{equation} \label{[|g_{k,s}^{(2)}|_infty]}
|g_{k,s}^{(2)}(x)| \ \le \ 2^{2k+5}\varphi_{k,s}(x), \ x \in \TT.
\end{equation} 
One can also verify that
\begin{equation} \label{[{supp}g_{k,s}]}
\operatorname{supp} g_{k,s} \ = \ I_{k,s} \ =: [2^{-k}s, 2^{-k}(s+1)], \quad 
\operatorname{int}I_{k,s} \cap \operatorname{int}I_{k,s'} \ = \ \varnothing, \ s \not= s',
\end{equation} 
and by the equation $\int_{0}^1 M_4(x)\, dx = 1$, 
\begin{equation} \label{[int g_{k,s}]}
\int_{0}^1 g_{k,s}(x)\, dx \ = \ 2^{-k-2}.
\end{equation} 
Let 
\begin{equation} \nonumber
g_{k,s}:= \ \prod_{j=1}^d g_{k_j,s_j}, \ k \in \ZZdp, \ s \in I^d(k).
\end{equation} 
Put for $u \subset [d]$,
$
D^2_u
:= \
\frac{\partial^{2|u|} }{\prod_{j \in u}\partial x^2_j}.
$
Then we have for $k \in \ZZdp, \ s \in I^d(k)$ and $u \subset [d]$,
\begin{equation} \label{[D^2_u g]}
|D^2_u g_{k,s}(x)| \ = \ 2^{5|u|} 2^{2|k_u|_1} 
\prod_{j \in u} \varphi_{k_j,s_j}(x_j)\, \prod_{j \not\in u} g_{k_j,s_j}(x_j), \ x \in \TTd.
\end{equation} 
Take $m,n \in \NN$ with $n > m$ and define the function 
\begin{equation} \label{[f_{m,n}]}
f_{m,n}:= \ 2^{-5d} \sum_{l=m}^n 2^{-\alpha l}\sum_{|k|_1 = l}\sum_{s \in I^d(k)} g_{k,s}.
\end{equation}
Observe that
\begin{equation} \nonumber
f_{m,n}= \ 2^{-5d} \sum_{l=m}^n 2^{-\alpha l} \sum_{|k|_1 = l} g_k,
\end{equation} 
where
\begin{equation} \label{[g_k=prod]}
g_k:= \ \prod_{j=1}^d  g_{k_j}, \ k \in \NNd.
\end{equation}
We will show that
\begin{equation} \label{[Delta^{2,u}_h (f_{m,n},x)]}
\Delta^{2,u}_h (f_{m,n},x) \ \le \ 
{\prod_{j \in u} |h_j|^\alpha}, \ x \in \TTd, \ h \in \RRd.
\end{equation}
Let us prove this inequality for $u=[d]$ and $h \in \RRdp$, the general case of $u$ can be proven in a similar way with a slight modification. For $h \in \RR$ and $k \in \NN$, let $h^{(k)} \in [0,2^{-k})$ is the number defined by $h = h^{(k)} + s 2^{-k}$ for some $s \in \ZZ$. For $h \in \RRd$ and $k \in \NNd$, put
 \begin{equation} \nonumber
h^{(k)}:= \ \prod_{j=1}^d h_j^{(k_j)}.
\end{equation} 
Since $g_k$ is  $2^{-k_j}$-periodic in variable $x_j$, we have 
 \begin{equation} \nonumber
\Delta^{2,[d]}_h (g_k,x) \ = \ 
\Delta^{2,[d]}_{h^{(k)}} (g_k,x), \ x \in \TTd, \ h \in \RRd.
\end{equation}
By using the formula 
\[
\Delta^2_h(f,x) \ = \  h^2\int_{\RR} (x+y)[h^{-1}M_2(h^{-1}y)]\, dy
\]
 for twice-differentiable function $f$ on $\RR$ (see, e.g., \cite[p.45]{DL}), and  \eqref{[|g_{k,s}^{(2)}|_infty]} we get
 \begin{equation} \nonumber
\Delta^{2,[d]}_h (g_{k,s},x) \ = \ 
\prod_{j=1}^d h_j^2\int_{\RRd}D^2_{[d]} \, g_{k,s}(y)
\prod_{j=1}^d  h_j^{-1} M_2(h_j^{-1}(y_j-x_j))\, dy,
\end{equation}
and therefore,
\begin{equation} \nonumber
\begin{split}
\Delta^{2,[d]}_h (g_k,x) 
\ &= \
 \sum_{s \in I^d(k)} \Delta^{2,[d]}_{h^{(k)}} (g_{k,s},x) \\[1.5ex]
\ &= \
 \sum_{s \in I^d(k)} \prod_{j=1}^d (h^{(k)}_j)^2\int_{\RRd}D^2_{[d]} \, g_{k,s}(y)
\prod_{j=1}^d  (h^{(k)}_j)^{-1} M_2((h^{(k)}_j)^{-1}(y_j-x_j))\, dy.
\end{split}
\end{equation}
Hence,
\begin{equation} \nonumber
\begin{split}
|\Delta^{2,[d]}_h (g_k,x)| 
\ &\le \
 \sum_{s \in I^d(k)} \int_{\RRd}  5^d 2^{|k|_1} \varphi_{k,s}(y)
\prod_{j=1}^d (h^{(k)}_j)^2 \prod_{j=1}^d  (h^{(k)}_j)^{-1} M_2((h^{(k)}_j)^{-1}(y_j-x_j))\, dy \\[1.5ex]
\ &\le \
 2^{5d} 2^{|k|_1} \int_{\RRd}  \sum_{s \in I^d(k)} \varphi_{k,s}(y)
\prod_{j=1}^d (h^{(k)}_j)^\alpha \prod_{j=1}^d (2^{-k_j}h^{(k)}_j)^{2-\alpha} \prod_{j=1}^d  (h^{(k)}_j)^{-1} M_2((h^{(k)}_j)^{-1}(y_j-x_j))\, dy \\[1.5ex]
\ &\le \
 2^{5d} 2^{|k|_1} \prod_{j=1}^d h_j^\alpha \int_{\RRd}  \sum_{s \in I^d(k)} \varphi_{k,s}(y)
 \prod_{j=1}^d  (h^{(k)}_j)^{-1} M_2((h^{(k)}_j)^{-1}(y_j-x_j))\, dy.
\end{split}
\end{equation}
By the inequalities \eqref{[D^2_u g]} and 
 \begin{equation} \nonumber
 \sum_{l=m}^n \ \sum_{|k|_1 = l} \ \sum_{s \in I^d(k)} \varphi_{k,s}(y)
\ \le \ 1, \ y  \in \RRd,
\end{equation}
we derive that
\begin{equation} \nonumber
\begin{split}
|\Delta^{2,[d]}_h (f_{m,n},x)| 
\ &\le \
2^{-5d} \sum_{l=m}^n 2^{-\alpha l} \sum_{|k|_1 = l} |\Delta^{2,[d]}_h (g_k,x)| \\[1.5ex]
\ &\le \
2^{-5d} \sum_{l=m}^n \sum_{|k|_1 = l} 2^{-\alpha l} 
 2^{5d} 2^{|k|_1} \prod_{j=1}^d h_j^\alpha \int_{\RRd}  \sum_{s \in I^d(k)} \varphi_{k,s}(y)
 \prod_{j=1}^d  (h^{(k)}_j)^{-1} M_2((h^{(k)}_j)^{-1}(y_j-x_j))\, dy \\[1.5ex]
\ &\le \
\prod_{j=1}^d h_j^\alpha  \int_{\RRd}  \sum_{l=m}^n \sum_{|k|_1 = l}\sum_{s \in I^d(k)} \varphi_{k,s}(y)
 \prod_{j=1}^d  (h^{(k)}_j)^{-1} M_2((h^{(k)}_j)^{-1}(y_j-x_j))\, dy \\[1.5ex]
\ &\le \
\prod_{j=1}^d h_j^\alpha  \int_{\RRd}\,
\prod_{j=1}^d  (h^{(k)}_j)^{-1} M_2((h^{(k)}_j)^{-1}(y_j-x_j))\, dy \\[1.5ex]
\ &= \
\prod_{j=1}^d h_j^\alpha   
\int_{\RRd}\,\prod_{j=1}^d M_2(y_j)\, dy 
\ \le \
\prod_{j=1}^d h_j^\alpha. 
\end{split}
\end{equation}
The inequality \eqref{[Delta^{2,u}_h (f_{m,n},x)]} is proven.
This means that
$
f_{m,n} \in \Uas.
$
From \eqref{[g_k=prod]}, \eqref{[{supp}g_{k,s}]}, 
\eqref{[int g_{k,s}]} we have that if $m$ is given then for arbitrary $n \ge m$, 
\begin{equation} \label{[f_{m,n}>]}
\|f_{m,n}\|_p
\ \ge \
\|f_{m,n}\|_1
 \ = \ 
\ 2^{-5d} \sum_{l=m}^n 2^{-\alpha l} \sum_{|k|_1 = l}  \|g_k\|_1
 \ = \ 
\ 2^{-7d} \sum_{l=m}^n 2^{-\alpha l} \binom{l-1}{d-1}.
\end{equation}
On the other hand, one can verify that
$
f_{m,n} (\xi) \ = 0,  \ \xi \in \mathring{G}(m)
$ which yields that $\mathring{S}_m(\mathring{\Phi}_m,f_{m,n}) = 0$ for arbitrary $\mathring{\Phi}_m$, and consequently, by 
\eqref{[f_{m,n}>]} for arbitrary $n \ge m$,
\begin{equation} \nonumber  
\mathring{s}_m(\Uas)_p
\ \ge \ 
\|f_{m,n} - \mathring{S}_m(\Phi,f_{m,n})\|_p
\ = \ 
 \|f_{m,n}\|_p
\ \ge \ 
2^{-7d} \sum_{l=m}^n 2^{-\alpha l} \binom{l-1}{d-1}
\end{equation}
This means that for $t = 2^{-\alpha}$,
\begin{equation}  \nonumber
\mathring{s}_m(\Uas)_p
\ \ge \  
2^{-7d} \sum_{l=m}^\infty 2^{-\alpha l} \binom{l-1}{d-1}
\ = \
2^{-7d}\, 2^{-\alpha m}\, F_{m-1,d-1}(t).
\end{equation}
Therefore, applying \eqref{[<F_{m,n}<]} gives
\begin{equation}  \nonumber
\begin{split}
\mathring{s}_m(\Uas)_p 
\ &> \
2^{-7d}\, 2^{-\alpha m}\, t F_{m,d-1}(t) 
\ = \ 
2^{-7d}\, 2^{-\alpha m}\, \beta(\alpha,d,m) \\
\ &> \
2^{-7d}\,2^{-\alpha m}\,(2^\alpha - 1)^{-d}\, \binom{m}{d-1}\, (2^\alpha - 1)^{d-1} \\
\ &= \ 
(2^\alpha - 1)^{-1} \, 2^{-7d}\, 2^{- \alpha m} \, \binom{m}{d-1}
\end{split}
\end{equation}
which proves \eqref{LowerBnd,Has(1)]}. 
\end{proof}

\begin{corollary} \label{corollary[|f - R_m(f)|_p,Has]}
Let  $0 < p \le \infty$ and $0 < \alpha \le 2$. Then we have for every
$m \ge d$, 
\begin{equation}  \nonumber
 2^{-7d}\,  \beta(d,m)\, 2^{- \alpha m}
 \ \le \
\mathring{s}_m(\Uas)_p
\ \le \  
 [2(p+1)^{1/p}]^{-d}\, \beta(d,m)\, 2^{- \alpha m}. 
\end{equation}
Moreover, if in addition, $m \ge 2(d-1)$,
\begin{equation}  \nonumber
(2^\alpha - 1)^{-1}\, 2^{-7d} \, 2^{- \alpha m} \, \binom{m}{d-1}
\ \le \
\mathring{s}_m(\Uas)_p
\ \le \ 
 a^\circ(d)\, 2^{- \alpha m} \, \binom{m}{d-1}. 
\end{equation}
\end{corollary}

Put
\[
\gamma'(\nu,m)
\ = \
\gamma'(\alpha,\nu,m)
:= \
 \sum_{l=0}^\nu \binom{\nu}{l}\, 2^{-7l} \,\beta(\alpha,l,m). 
\]

\begin{theorem} \label{LowerBnd,Ha]}
Let  $1\le p \le \infty$, $0 < \alpha \le 2$ and $1 \le \nu \le d$. Then we have every $m \ge \nu$, 
\begin{equation}  \label{LowerBnd,Ha(1)]}
s^\nu_m(\Uan)_p
 \ge   
\gamma'(\nu,m) \, 2^{-\alpha m}
 >  
[2^7(2^\alpha - 1)]^{-1}\, \biggl(2^\alpha - \frac{127}{128}\biggl)^{\nu-1} \, 2^{-\alpha m}\, \binom{m}{\nu-1}. 
\end{equation}
\end{theorem}

\begin{proof}
We first consider the case $\nu = d$. We take the univariate functions $g_{k,s}$, $g_k$ as in \eqref{[g_{k,s}]}, and for every $u \subset [d]$, define the functions
\begin{equation} \nonumber
g_{k,s}^u:= \ \prod_{j \in u} g_{k_j,s_j}, \ k \in \ZZdp, \ s \in I^d(k),
\quad
g_k^u:= \ \prod_{j \in u} g_{k_j}.
\end{equation} 
If $m,n \in \NN$ with $n > m$, and $u \subset [d]$, we define the $|u|$-variate function 
\begin{equation} \nonumber
f_{m,n}^u:= \ 2^{-5|u|} \sum_{l=m}^n 2^{-\alpha l}
\sum_{|k|_1 = l}\, \sum_{s \in I^d(k)} g_{k,s}^u,
\end{equation}
and the $d$-variate function
\begin{equation} \nonumber
\phi_{m,n}:= \sum_{u \subset [d]} f_{m,n}^u.
\end{equation}
From the proof of Theorem \ref{LowerBnd,Has]} we can see that
\begin{equation} \label{[phi-in-Ua]}
\phi_{m,n} \in \Ua,
\quad
\phi_{m,n} (\xi) \ = 0,  \ \xi \in G^\nu(m).
\end{equation}

Observe that
\begin{equation} \nonumber
\phi_{m,n}:= \ \sum_{u \subset [d]} 2^{-5|u|} \sum_{l=m}^n 2^{-\alpha l} \sum_{|k|_1 = l} g_k^u
\end{equation}

 From \eqref{[g_k=prod]}, \eqref{[{supp}g_{k,s}]}, 
\eqref{[int g_{k,s}]} we can also verify that if $m$ is given then for arbitrary $n > m$, 
\begin{equation} \nonumber
\begin{split}
\|\phi_{m,n}\|_p
\ &\ge \
\|\phi_{m,n}\|_1 
 \ = \ 
\ \sum_{u \subset [d]} 2^{-5|u|} \sum_{l=m}^n 2^{-\alpha l} \sum_{|k|_1 = l}  \|g_k^u\|_1 \\[1.5ex]
 \ &= \ 
\ \sum_{u \subset [d]} 2^{-7|u|}  \sum_{l=m}^n 2^{-\alpha l} \binom{l-1}{|u|-1}
 \ = \ 
\ \sum_{s=1} \binom{d}{s} 2^{-7s}  \sum_{l=m}^n 2^{-\alpha l} \binom{l-1}{s-1}\\[1.5ex]
\end{split}
\end{equation}
By \eqref{[phi-in-Ua]} we have $S^d_m(\Phi^d_m,\phi_{m,n}) = 0$ for arbitrary $\Phi^d_m$, and consequently, by 
\eqref{[f_{m,n}>]} for arbitrary $n > m$,
\begin{equation}  \label{[s_m(Ua)_p>](1)}
s^d_m(\Ua)_p
\ \ge \  
 \|\phi_{m,n}\|_p
\ \ge \ 
\sum_{s=1}^d \binom{d}{s} 2^{-7s}  \sum_{l=m}^n 2^{-\alpha l} \binom{l-1}{s-1}.
\end{equation}
This means that for $t = 2^{-\alpha}$,
\begin{equation}  \label{[s_m(Ua)_p>](2)}
s^d_m(\Ua)_p
\ \ge \  
\sum_{s=1}^d \binom{d}{s} 2^{-7s}  \sum_{l=m}^\infty 2^{-\alpha l} \binom{l-1}{s-1}
\ = \
\sum_{s=1}^d \binom{d}{s} 2^{-7s} \, F_{m-1,s-1}(t).
\end{equation}
Therefore, applying \eqref{[<F_{m,n}<]} gives
\begin{equation}  \label{LowerBnd,Ha(3)]}
\begin{split}
s^d_m(\Ua)_p 
\ &> \
\sum_{s=1}^d \binom{d}{s} 2^{-7s} \, 2^{-\alpha m}\, t F_{m,s-1}(t) \\
\ &= \ 
2^{-\alpha m}\, \gamma'(\alpha,d,m) \\
\ &> \
2^{-\alpha m}\, \sum_{s=1}^d \binom{d}{s} 2^{-7s}\,(2^\alpha - 1)^{-s}\, 
\sum_{l=0}^{s-1} \binom{m}{l}\, (2^\alpha - 1)^l \\
\ &= \ 
2^{-\alpha m}\, \sum_{s=0}^{d-1} \binom{d}{s+1} 2^{-7(s+1)}\,(2^\alpha - 1)^{-(s+1)}\, 
\sum_{l=0}^s \binom{m}{l}\, (2^\alpha - 1)^l \\
\ &= \ 
2^{-\alpha m}\, \sum_{l=0}^{d-1} \binom{m}{l}\, (2^\alpha - 1)^l \,
\sum_{s=0}^l \binom{d}{s+1} 2^{-7(s+1)}\,(2^\alpha - 1)^{-(s+1)} \\
\ &> \
2^{-\alpha m}\,  \binom{m}{d-1}\, (2^\alpha - 1)^{d-1} \,
\sum_{s=0}^l \binom{d-1}{s} 2^{-7(s+1)}\,(2^\alpha - 1)^{-(s+1)} \\
\ &= \ 
[2^7(2^\alpha - 1)]^{-1}\, 2^{-\alpha m}\,  \binom{m}{d-1}\, (2^\alpha - 1)^{d-1} \,
\sum_{s=0}^{d-1} \binom{d-1}{s} [2^7(2^\alpha - 1)]^{-s} \\
\ &= \ 
[2^7(2^\alpha - 1)]^{-1}\, (2^\alpha - 1)^{d-1} \,
[1 + 2^7(2^\alpha - 1)^{-1}]^{d-1} \, 2^{-\alpha m}\, \binom{m}{d-1}\\
\ &= \ 
[2^7(2^\alpha - 1)]^{-1}\, \biggl(2^\alpha - \frac{127}{128}\biggl)^{d-1} \,
 2^{-\alpha m}\, \binom{m}{d-1}\\
\end{split}
\end{equation}
which proves \eqref{LowerBnd,Ha(1)]} for the case $\nu = d$.

To process the case $\nu < d$ we take the function $\phi^{[\nu]}_{m,n} \in \Uan$ by
\begin{equation} \nonumber
\phi^{[\nu]}_{m,n}:= \ \sum_{u \subset [\nu]} f_{m,n}^u.
\end{equation}
Similarly to \eqref{[s_m(Ua)_p>](1)} we have
\begin{equation}  \nonumber
s^\nu_m(\Uan)_p
\ \ge \  
 \|\phi^{[\nu]}_{m,n}\|_1
\ = \ 
\sum_{s=1}^\nu \binom{\nu}{s} 2^{-7s}  \sum_{l=m}^n 2^{-\alpha l} \binom{l-1}{s-1}.
\end{equation} 
Hence, by replacing $d$ by $\nu$ in the estimations  \eqref{[s_m(Ua)_p>](2)} and 
\eqref{LowerBnd,Ha(3)]} we prove the case $\nu < d$.
\end{proof}

\begin{corollary} \label{corollary[|f - R^nu_m(f)|_p]}
Let $0 < p \le \infty$, $0 < \alpha \le 2$ and $1\le \nu \le d$. Then we have every $m \ge \nu$, 
\begin{equation}  \nonumber
 2^{- \alpha m} \sum_{l=0}^\nu \binom{\nu}{l}\, 2^{-7l}  \beta(l,m)
\le
s^\nu_m(\Uan)_p
\le
 2^{- \alpha m} \sum_{l=0}^\nu \binom{\nu}{l}\, [2(p+1)^{l/p}]^{-l} \beta(l,m).
\end{equation}
Moreover,  if in addition, $m \ge 2(\nu-1)$,
\begin{equation}  \nonumber
2^7(2^\alpha - 1)]^{-1}\, \biggl(2^\alpha - \frac{127}{128}\biggl)^{\nu-1} \, 2^{-\alpha m}\, \binom{m}{\nu-1}
\le 
s^\nu_m(\Uan)_p
\le  
a(\nu)\,  2^{- \alpha m}\, \binom{m}{\nu-1}.
\end{equation}
\end{corollary}

\subsection{Cubature formulas}
We are interested in  cubature formulas on Smolyak grids for approximately computing of the integral 
\begin{equation} \nonumber
I(f) :=
\int_{[0,1]^d} f(x)\,dx.
\end{equation}

If $f \in \Uas$, we use  the cubature formula on grids $\mathring{G}(m)$ given by
\begin{equation*}
\mathring{I}_m(\mathring{\Lambda}_m, f)
\ = \
\sum_{ \xi \in \mathring{G}(m)} \lambda_{\xi} f(\xi),
\end{equation*}
where $\mathring{\Lambda}_m = (\lambda_{\xi})_{\xi\in \mathring{G}(m)}$ are integration weights.
The quantity of optimal cubature $\mathring{\Int}_m(F_d)$ on Smolyak grids $\mathring{G}(m)$ is defined by
\begin{equation} \nonumber
\mathring{\Int}_m(F_d)
\ := \ \inf_{\mathring{\Lambda}_m} \  \sup_{f \in F_d} \, 
|I(f) - \mathring{I}_m(\mathring{\Lambda}_m, f)|.
\end{equation}

For a family $\mathring{\Phi}_m = \{\varphi_\xi\}_{\xi \in \mathring{G}(m)}$ of functions on $\TTd$, the linear sampling algorithm $\mathring{S}_m(\mathring{\Phi}_m ,\cdot)$ generates the cubature formula 
$\mathring{I}_m(\mathring{\Lambda}_m, f)$ on Smolyak grid $\mathring{G}(m)$ by
\begin{equation} \nonumber
\mathring{I}_m(\mathring{\Lambda}_m, f)
\ = \
\sum_{\xi \in \mathring{G}(m)} \lambda_\xi f(\xi),
\end{equation}
where the integration weights $\mathring{\Lambda}_m = (\lambda_{\xi})_{\xi\in \mathring{G}(m)}$ are given by 
$
\lambda_\xi = \int_{\IId} \varphi_\xi(x) \ dx.
$
Hence, we have
$
 |I(f) - \mathring{I}_m(\mathring{\Lambda}_m, f)|
 \ \le \
 \|f - \mathring{S}_m(\mathring{\Phi}_m,f)\|_1
$
and consequently,
\begin{equation} \label{[i_n<s_n(s)]}
\mathring{\Int}_m(F_d)
\ \le \
\mathring{s}_m(F_d)_1.
\end{equation}

\begin{theorem} \label{theorem[Int_m<Has]}
Let $0 < \alpha \le 2$. Then we have for every
every $m \ge d$, 
\begin{equation}  \label{[Int_m<Has](1)}
 2^{-7d}\,  \beta(d,m)\, 2^{- \alpha m} \beta(d,m)
\ \le \  
\mathring{\Int}_m^s(\Uas) 
\ \le \  
 2^{-2d}\,  \beta(d,m)\, 2^{- \alpha m}. 
\end{equation}
Moreover, if in addition, $m \ge 2(d-1)$,
\begin{equation}  \label{[Int_m<Has](2)}
(2^\alpha -1)^{-1} 2^{-7d}\, 2^{- \alpha m} \, \binom{m}{d-1}
\ \le \
\mathring{\Int}_m(\Uas)
\ \le \  
 a^\circ(d)\, 2^{-2d}\, 2^{- \alpha m} \, \binom{m}{d-1}, 
\end{equation}
where 
\begin{equation} \nonumber
a^\circ(d)
:= \ |2^\alpha -2|^{- |\operatorname{sgn}(\alpha - 1)|} \times
\begin{cases}
1, \ & \alpha > 1,\\[1.5ex]
d, \ & \alpha = 1,\\[1.5ex]
(2^\alpha -1)^{-d}, \ & \alpha < 1.
\end{cases}
\end{equation}
\end{theorem}

\begin{proof}
The upper bounds in \eqref{[Int_m<Has](1)}--\eqref{[Int_m<Has](2)} follow from \eqref{[i_n<s_n(s)]} and 
Theorem \ref{theorem[|f - R_m(f)|_p,Has]}. 
To prove the lower bounds, we take the function $f_{m,n} \in \Uas$ as in \eqref{[f_{m,n}]} with the property 
$
f_{m,n} (\xi) \ = 0,  \ \xi \in \mathring{G}(m).
$
Notice that $f_{m,n}$ is a nonnegative function. Hence, we have 
$\mathring{\Lambda}^s_m(\mathring{\Phi}_m,f_{m,n}) = 0$ for arbitrary $\mathring{\Phi}_m$, and consequently, by 
\eqref{[f_{m,n}>]} for arbitrary $n \ge m$,
\begin{equation}  \nonumber
\mathring{\Int}_m(\Uas)
\ \ge \
|f_{m,n} -  \mathring{\Lambda}^s_m(\mathring{\Phi}_m,f_{m,n})| 
\ = \ 
 \|f_{m,n}\|_1
\ \ge \ 
2^{-7d} \sum_{l=m}^n 2^{-\alpha l} \binom{l-1}{d-1}.
\end{equation}
Comparing with \eqref{LowerBnd,Has(1)]}, we can see that $\mathring{\Int}_m^s(\Uas)$ can be estimated from below as in the proof of Theorem \eqref{LowerBnd,Has]} for $\mathring{s}_m^s(\Uas)$. This proves the lower bounds 
in \eqref{[Int_m<Has](1)}--\eqref{[Int_m<Has](2)}.  
\end{proof}

If $f \in \Uan$, we use  the cubature formula on grids $G^\nu(m)$ given by
\begin{equation*}
I^\nu_m(\Lambda^\nu_m, f)
\ = \
\sum_{ \xi \in G^\nu(m)} \lambda_{\xi} f(\xi),
\end{equation*}
where $\Lambda^\nu_m = (\lambda_{\xi})_{\xi\in G^\nu(m)}$ are integration weights.
The quantity of optimal cubature $\Int^\nu_m(F_d)$ on Smolyak grids $G^\nu(m)$ is introduced by
\begin{equation} \nonumber
\Int^\nu_m(F_d)
\ := \ \inf_{\Lambda_m} \  \sup_{f \in F_d} \, |I(f) - I^\nu_m(\Lambda_m, f)|.
\end{equation}
Put
\[
\gamma(\alpha,\nu,m)
:= \
 \sum_{l=0}^\nu \binom{\nu}{l}\, 4^{-l} \, \beta(\alpha,l,m). 
\]
In a similar way to Theorem~\ref{theorem[Int_m<Has]} we obtain

\begin{theorem} \label{theorem[Int_m<Ha]}
Let $0 < \alpha \le 2$ and $1 \le \nu \le d$. Then we have for every
every $m \ge \nu$,
\begin{equation}  \nonumber
 2^{- \alpha m} \sum_{l=0}^\nu \binom{\nu}{l}\, 2^{-7l}  \beta(l,m)
\ \le \
\Int_m^\nu(\Uan)
\ \le \
 2^{- \alpha m} \sum_{l=0}^\nu \binom{\nu}{l}\, 2^{-2l} \beta(l,m).
\end{equation}
Moreover, if in addition, $m \ge 2(\nu-1)$,
\begin{equation}  \nonumber
[2^7(2^\alpha - 1)]^{-1}\, \biggl(2^\alpha - \frac{127}{128}\biggl)^{\nu-1}
   2^{-\alpha m}\, \binom{m}{\nu-1}
\ \le \
\Int^\nu_m(\Uan)
\ \le \  
 a(\nu)\,   2^{- \alpha m} \, \binom{m}{\nu-1}, 
\end{equation}
where 
\begin{equation} \nonumber
a(\nu)
:= \ |2^\alpha -2|^{- |\operatorname{sgn}(\alpha - 1)|} \times
\begin{cases}
(5/4)^\nu, \ & \alpha > 1,\\[1.5ex]
\nu\,(5/4)^\nu, \ & \alpha = 1,\\[1.5ex]
[1 + 1/4(2^\alpha -1)]^\nu, \ & \alpha < 1.
\end{cases}
\end{equation}
\end{theorem}

\section{Sampling recovery based on B-spine quasi-interpolation}
 \label{Sampling recovery[alpha>2]}
 
In this section, we extend the results on upper bounds of $s^\nu_m(\Uan)_p$ and $\Int_m^\nu(\Uan)$ 
in Subsection~\ref{Upper bounds, alpha<2} to arbitrary smoothness $\alpha > 0$. To this end we construct B-spline quasi-interpolation representations for  continuous functions on $\TTd$, find an explicit formula for the coefficients functionals and estimate the quasi-norm of $L_p(\TTd)$ for the component functions in these representations for functions in $\Ha$.
 
\subsection{B-spline quasi-interpolation representations} 
\label{Quasi-interpolant}

In order to construct B-spline quasi-interpolation representations for continuous functions on $\TTd$, we preliminarily introduce quasi-interpolation operators for functions on $\RRd$. For a given natural number $\ell,$ denote by  $M_\ell$ the cardinal B-spline of order $\ell$ with support $[0,\ell]$ and 
knots at the points $0, 1,...,\ell$. 
We fixed $r \in \NN$ and take the cardinal B-spline $M:= M_{2r}$ of even order $2r$.
Let $\Lambda = \{\lambda(s)\}_{|j| \le \mu}$ be a given finite even sequence, i.e., 
$\lambda(-j) = \lambda(j)$ for some $\mu \ge r - 1$. 
We define the linear operator $Q$ for functions $f$ on $\RR$ by  
\begin{equation} \label{def:Q}
Q(f,x):= \ \sum_{s \in \ZZ} \Lambda (f,s)M(x-s), 
\end{equation} 
where
\begin{equation} \label{def:Lambda}
\Lambda (f,s):= \ \sum_{|j| \le \mu} \lambda (j) f(s-j + r).
\end{equation}
The operator $Q$ is local and bounded in $C(\RR)$  (see \cite[p. 100--109]{C92}).
An operator $Q$ of the form \eqref{def:Q}--\eqref{def:Lambda} is called a 
{\it quasi-interpolation operator in} $C(\RR)$ if  it reproduces 
$\Pp_{2r-1}$, i.e., $Q(f) = f$ for every $f \in \Pp_{2r-1}$, where $\Pp_l$ denotes the set of $d$-variate polynomials of degree at most $l-1$ in each variable.

If $Q$ is a quasi-interpolation operator of the form 
\eqref{def:Q}--\eqref{def:Lambda}, for $h > 0$ and a function $f$ on $\RR$, 
we define the operator $Q(\cdot;h)$ by
$
Q(f;h) 
:= \ 
\sigma_h \circ Q \circ \sigma_{1/h}(f),
$
where $\sigma_h(f,x) = \ f(x/h)$.
Let $Q$ be a quasi-interpolation operator of the form \eqref{def:Q}--\eqref{def:Lambda} in $C({\RR}).$ 
If $k \in \ZZ_+ $, we introduce the operator $Q_k$  by  
\begin{equation*}
Q_k(f,x) := \ Q(f,x;h^{(k)}), \  x \in \RR, \quad h^{(k)}:= \ (2r)^{-1}2^{-k}. 
\end{equation*}
We define the integer translated dilation $M_{k,s}$ of $M$ by   
\begin{equation*}
M_{k,s}(x):= \ M(2r2^k x - s), \ k \in {\ZZ}_+, \ s \in \ZZ. 
\end{equation*}
Then we have for $k \in {\ZZ}_+$,
\begin{equation} \nonumber
Q_k(f)(x)  \ = \ 
\sum_{s \in \ZZ} a_{k,s}(f)M_{k,s}(x), \ \forall x \in \RR, 
\end{equation}
where the coefficient functional $a_{k,s}$ is defined by
\begin{equation} \label{def[a_{k,s}(f)]}
a_{k,s}(f):= \ \Lambda(f,s;h^{(k)}) 
= \   
\sum_{|j| \le \mu} \lambda (j) f(h^{(k)}(s-j+r)).
\end{equation}
Notice that $Q_k(f)$ can be written in the form:
\begin{equation} \label{[L_{k,s}-representation]}
Q_k(f)(x)  \ = \ 
\sum_{s \in \ZZ} f(h^{(k)}(s+r))L_k(x-s), \ \forall x \in \RR, 
\end{equation}
where the function $L_k$ is defined by
\begin{equation} \label{L_{k,s}}
L_k:= \  
= \   
\sum_{|j| \le \mu} \lambda (j) M_{k,j}.
\end{equation}
From \eqref{[L_{k,s}-representation]} and \eqref{L_{k,s}} we get for a function $f$ on $\RR$,
\begin{equation} \label{[|Q_k(f)|<]}
\|Q_k(f)\|_{C(\RR)} 
\ \le \ 
\|L_\Lambda\|_{C(\RR)} \|f\|_{C(\RR)}
\ \le \ 
\|\Lambda \|\|f\|_{C(\RR)},    
\end{equation}
where
\begin{equation} \label{[L]}
L_\Lambda(x):= \  
= \   
\sum_{s \in \ZZ} \sum_{|j| \le \mu} \lambda (j) M(x-j-s),
\quad
\|\Lambda \|= \ \sum_{|j| \le \mu} |\lambda (j)|.
\end{equation} 
 
For $k \in \ZZdp$, let the mixed operator $Q_k$ be defined by
\begin{equation} \label{def:Mixed[Q_k]} 
Q_k:= \prod_{i=1}^d  Q_{k_i},
\end{equation}
where the univariate operator
$Q_{k_i}$ is applied to the univariate function $f$ by considering $f$ as a 
function of  variable $x_i$ with the other variables held fixed.
We define the $d$-variable B-spline $M_{k,s}$ by
\begin{equation} \label{def:Mixed[M_{k,s}]}
M_{k,s}(x):=  \ \prod_{i=1}^d M_{k_i, s_i}( x_i),  
\ k \in {\ZZ}^d_+, \ s \in {\ZZ}^d,
\end{equation}
where ${\ZZ}^d_+:= \{s \in {\ZZ}^d: s_i \ge 0, \ i \in [d] \}$.
Then we have
\begin{equation*} 
Q_k(f,x)  \ = \ 
\sum_{s \in \ZZ^d} a_{k,s}(f)M_{k,s}(x), \quad \forall x \in \RRd, 
\end{equation*}
where $M_{k,s}$ is the mixed B-spline  defined in \eqref{def:Mixed[M_{k,s}]}, 
and
\begin{equation} \label{def:Mixed[a_{k,s}(f)]}
a_{k,s}(f) 
\ = \   
\biggl(\prod_{j=1}^d a_{k_j,s_j}\biggl)(f),
\end{equation}
and the univariate coefficient functional
$a_{k_i,s_i}$ is applied to the univariate function $f$ by considering $f$ as a 
function of  variable $x_i$ with the other variables held fixed. 

Since $M(2r\,2^k x)=0$ for every $k \in \ZZ_+$ and $x \notin (0,1)$, we can extend  the univariate B-spline  
$M(2r\,2^k\cdot)$ to an $1$-periodic function on the whole $\RR$. Denote this periodic extension by $N_k$ and define 
\begin{equation*}
N_{k,s}(x):= \ N_k(x - s), \ k \in {\ZZ}_+, \ s \in I(k), 
\end{equation*}
where $I(k) := \{0,1,..., 2r2^k - 1\}$.
We define the $d$-variable B-spline $N_{k,s}$ by
\begin{equation} \nonumber
N_{k,s}(x):=  \ \prod_{i=1}^d N_{k_i, s_i}( x_i),  \ k \in {\ZZ}^d_+, \ s \in I^d(k),
\end{equation}
where $I^d(k):=\prod_{i=1}^d I(k_i)$.
Then we have for functions $f$ on $\TTd$,
\begin{equation} \label{def[periodicQI]}
Q_k(f,x)  \ = \ 
\sum_{s \in I^d(k)} a_{k,s}(f) N_{k,s}(x), \quad \forall x \in \TTd. 
\end{equation}
Since the function $L_\Lambda$ defined in \eqref{[L]} is $1$-periodic, from \eqref{[|Q_k(f)|<]} it follows that for a function $f$ on $\TT$,
\begin{equation} \label{[|Q_k(f)|<(T)]}
\|Q_k(f)\|_{C(\TT)} 
\ \le \ 
\|L_\Lambda\|_{C(\TT)} \|f\|_{C(\TT)}
\ \le \ 
\|\Lambda \|\|f\|_{C(\TT)},    
\end{equation}

For $k \in \ZZdp$, we write  $k  \to  \infty$ if $k_i  \to  \infty$ for $i \in [d]$). In a way similar to the proof of \cite[Lemma 2.2]{Di13} one can show that  for every $f \in C(\TTd)$,
\begin{equation} \nonumber
\|f - Q_k(f)\|_{C(\TTd)}
\ \le \
C \sum_{u \subset [d], \ u \not= \varnothing} \omega_{2r}^u(f,2^{-k}),
\end{equation}
and, consequently,
\begin{equation} \label{ConvergenceMixedQ_k(f)}
\|f - Q_k(f)\|_{C(\TTd)} \to 0 , \ k  \to  \infty.
\end{equation}

For convenience we define the univariate operator $Q_{-1}$ by putting $Q_{-1}(f)=0$ for all $f$ on $\II$. Let  the operators $q_k$ be defined 
in the manner of the definition \eqref{def:Mixed[Q_k]} by
\begin{equation} \label{eq:Def[q_k]}
q_k\ := \ \prod_{i=1}^d \biggl(Q_{k_i}- Q_{k_i-1}\biggl), \ k \in {\ZZ}^d_+. 
\end{equation}

From the equation $Q_k = \sum_{k' \le k}q_{k'}$
and \eqref{ConvergenceMixedQ_k(f)} it is easy to see that 
a continuous function  $f$ has the decomposition
$
f \ =  \ \sum_{k \in {\ZZ}^d_+} q_k(f)
$
with the convergence in the norm of $C(\TTd)$.  
From  the refinement equation for the B-spline $M$, in the univariate case, we can represent the component functions $q_k(f)$ as 
 \begin{equation} \label{eq:RepresentationMixedq_k(f)}
q_k(f) 
= \ \sum_{s \in I^d(k)}c_{k,s}(f) N_{k,s},
\end{equation}
where $c_{k,s}$ are certain coefficient functionals of 
$f$. In the multivariate case, the representation  \eqref{eq:RepresentationMixedq_k(f)} holds true 
with the $c_{k,s}$ which are defined in the manner of the definition  
\eqref{def:Mixed[a_{k,s}(f)]} by
\begin{equation} \nonumber
c_{k,s}(f) 
\ = \   
\biggl(\prod_{j=1}^d c_{k_j,s_j}\biggl)(f).
\end{equation}
See \cite{Di11} for details. Thus, we have proven the following periodic B-spline quasi-interpolation representation for continuous functions on $\TTd$.
\begin{lemma} \label{lemma[representation]}
Every continuous function $f$ on $\TTd$ is represented as B-spline series 
\begin{equation} \label{eq:B-splineRepresentation}
f \ = \sum_{k \in {\ZZ}^d_+} \ q_k(f) = 
\sum_{k \in {\ZZ}^d_+} \sum_{s \in I^d(k)} c_{k,s}(f)N_{k,s}, 
\end{equation}  
converging in the norm of $C(\TTd)$, where the coefficient functionals $c_{k,s}(f)$ are explicitly constructed as
linear combinations of at most $m_0$ of function
values of $f$ for some $m_0 \in \NN$ which is independent of $k,s$ and $f$.
\end{lemma}

\subsection{A formula for the coefficients in B-spline quasi-interpolation representations} 
\label{Coefficient}

In this subsection, we find a explicit formula for the coefficients 
$c_{k,s}(f)$ and hence, estimate the quasi-norm of $\|q_k(f)\|_q$ of the component functions in the periodic B-spline quasi-interpolation representations \eqref{eq:B-splineRepresentation}. 

If $h \in \RRd$, we define the shift operator $T_h^s$ for functions $f$ on $\TTd$ by
$
T_h(f) := \
f(\cdot + h).
$
Recall that a $d$-variate Laurent polynomial is call a function $P$ of the form
\begin{equation} \label{def[P]}
P(z) = \
\sum_{s \in A}  c_s z^s,
\end{equation}
where $A$ is a finite subset in $\ZZd$ and $z^s:= \prod_{j=1}^d z_j^{s_j}$.
A $d$-variate Laurent polynomial $P$ as \eqref{def[P]} generates the operator $T_h^{[P]}$ by 
\begin{equation} \label{def[T_h^{[P]}]}
T_h^{[P]}(f) = \
\sum_{s \in A}  c_s T_{sh}(f).
\end{equation}
Sometimes we also write $T_h^{[P]}=T_h^{[P(z)]}$.
Notice that any operation over polynomials generates a corresponding operation over operators $T_h^{[P]}$. Thus, in particular, we have 
\begin{equation*}
T_h^{[a_1P_1 + a_2P_2]}(f) \ = \
a_1 T_h^{[P_1]}(f) + a_2 T_h^{[P_2]}(f), \quad T_h^{[P_1.P_2]}(f) \ = \ T_h^{[P_1]}\circ T_h^{[P_2]}(f).
\end{equation*}
By definitions we have 
\begin{equation*}
\Delta_h^l
= \
T_h^{[D_l]}, \  D_l:= \prod_{j=1}^d (z_j - 1)^l, \quad
\Delta_h^{l,u}
= \
T_h^{[D_{l,u}]}, \  D_{l,u}:= \prod_{j \in u}(z_j - 1)^l.
\end{equation*}
We say that a $d$-variate polynomial is a {\em tensor product polynomial} if it is of the form
$
P(z) = \
\prod_{j=1}^d P_j(z_j),
$
where $P_j(z_j)$ are univariate polynomial in variable $z_j$. 

\begin{lemma} \label{lemma[factor]}
Let $P$ be a tensor product Laurent polynomial, $h \in \RRd$ with $h_j \not=0$, and $l \in \NN$. Assume that $T_h^{[P]}(g)=0$ for every polynomial $g \in \Pp_{l-1}$ Then $P$ has a factor $D_l$ and consequently,
\begin{equation*}
T_h^{[P]}
\ = \
T_h^{[P^*]} \circ \Delta_h^l, \quad  P(z)= D_lP^*(z),
\end{equation*}
where $P^*$ is a tensor product Laurent polynomial.
\end{lemma}

\begin{proof} By the tensor product argument it is enough to prove the lemma for the case $d=1$. We prove this case by induction on $l$. Let $P(z)=\sum_{s = -m}^n  c_s z^s$ for some $m,n \in \ZZ_+$. Consider first the case $l=1$. Assume that $T_h^{[P]}(g)=0$ for every constant functions $g$. Then 
replacing by $g_0=1$ in \eqref{def[T_h^{[P]}]} we get
$
T_h^{[P]}(g_0)
\ = \ 
\sum_{s = -m}^n  c_s \ = \ 0.
$
By B\'ezout's theorem $P$ has a factor $(z-1)$. This proves the lemma for $l=1$. Assume it is true for $l-1$ and $T_h^{[P]}(g)=0$ for every polynomial $g$ of degree at most $l-1$. By the induction assumption we have 
\begin{equation} \label{[induction]}
T_h^{[P]}
\ = \
T_h^{[P_1]} \circ \Delta_h^{l-1}, \quad  P(z):= (z-1)^{l-1} P_1(z).
\end{equation}
We take a proper polynomial $g_l$ of degree $l-1$ (with the nonzero eldest coefficient). Hence 
$\psi_l = \Delta_h^{l-1}(g_l) = a$ where $a$ is a nonzero constant. Similarly to the case $l=1$, from the equations
$ 0 \ = \ T_h^{[P]} (g_l) \ = \ T_h^{[P_1]} (\psi_l)$ we conclude that $P_1$ has a factor $(z-1)$. Hence, by 
\eqref{[induction]} we can see that $P$ has a factor $(z-1)^l$. The lemma is proved.
\end{proof}

Let us return to the definition of quasi-interpolation operator $Q$ of the form 
\eqref{def:Q} induced by the sequence $\Lambda$  as in \eqref{def:Lambda} which can be uniquely characterized by the univariate symmetric Laurent polynomial
\begin{equation} \nonumber
P_\Lambda(z)
:= \ 
z^r \sum_{|s|\le \mu}\lambda (s) z^s.
\end{equation}
Let the $d$-variate symmetric tensor product Laurent polynomial $P_\Lambda$ be given by 
\begin{equation} \nonumber
P_\Lambda(z)
:= \ 
z^r \prod_{j=1}^d \sum_{|s_j|\le \mu}\lambda (s_j) z^j.
\end{equation}
For the periodic quasi-interpolation operator 
$
Q_k(f)  \ = \ 
\sum_{s \in I^d(k)} a_{k,s}(f) N_{k,s} 
$
given as in \eqref{def[periodicQI]}, from \eqref{def[a_{k,s}(f)]} we get
\begin{equation} \label{eq[a_{k,s}(f)]}
a_{k,s}(f)
\ = \ 
T_{h^{(k)}}^{[P_\Lambda]}(f)(sh^{(k)}).
\end{equation}
Let us find an explicit formula for the univariate operator $q_k(f)$.
We have for $k >0$,
\begin{equation*}
\begin{aligned}
Q_k(f)
\ & = \
\sum_{s \in I(k)} T_{h^{(k)}}^{[P_\Lambda]}(f)(sh^{(k)})) N_{k,s} \\[1.5ex]
\ & = \
\sum_{s \in I(k-1)} T_{h^{(k)}}^{[P_\Lambda]}(f)(2sh^{(k)}))  N_{k,2s} 
\  +  \
\sum_{s \in I(k-1)} T_{h^{(k)}}^{[P_\Lambda]}(f)((2s+1)h^{(k)})) 
 N_{k,2s+1}. 
\end{aligned}
\end{equation*}
From \eqref{eq[a_{k,s}(f)]} and the refinement equation for $M$, we deduce that 
\begin{equation*}
\begin{aligned}
Q_{k-1}(f)
\ & = \
\sum_{s \in I(k-1)} T_{h^{(k-1)}}^{[P_\Lambda]}(f)(sh^{(k-1)})) 
\biggl[2^{-2r+1}\sum_{j =0}^{2r}\binom{2r}{j} N_{k,2s+j}\biggl]\\[1.5ex]
\ & = \
2^{-2r+1}\sum_{j =0}^r\binom{2r}{2j}\sum_{s \in I(k-1)} T_{h^{(k-1)}}^{[P_\Lambda]}(f)(sh^{(k-1)})) 
 N_{k,2s+2j} \\[1.5ex]
\ & +  
2^{-2r+1}\sum_{j =0}^{r-1}\binom{2r}{2j+1}\sum_{s \in I(k-1)} T_{h^{(k-1)}}^{[P_\Lambda]}(f)(sh^{(k-1)})) 
 N_{k,2s+2j+1} \\[1.5ex]
\ & =: \
Q_{k-1}^{\operatorname{even}}(f) + Q_{k-1}^{\operatorname{odd}}(f).
\end{aligned}
\end{equation*}
By the identities $h^{(k-1)}= 2 h^{(k)}$, $N_{k,2r2^k +m} = N_{k,m}$ and $f(h^{(k)})(2r2^k + m)= f(h^{(k)}m)$ for 
$k \in \ZZ_+$ and $m \in \ZZ$, we have
\begin{equation*}
\begin{aligned}
Q_{k-1}^{\operatorname{even}}(f)
\ & = \
2^{-2r+1}\sum_{j =0}^r\binom{2r}{2j}\sum_{s \in j+I(k-1)} T_{h^{(k)}}^{[P_\Lambda]}(f)(2(s-j)h^{(k)})) 
 N_{k,2s} \\[1.5ex]
\ &= \  
2^{-2r+1}\sum_{j =0}^r\binom{2r}{2j}\sum_{s \in I(k-1)} T_{h^{(k)}}^{[P_\Lambda]}(f)(2(s-j)h^{(k)})) 
 N_{k,2s} \\[1.5ex]
\ &= \  
\sum_{s \in I(k-1)} T_{h^{(k)}}^{[P_{\operatorname{even}}']}(f)(2sh^{(k)}) N_{k,2s}, \\
\end{aligned}
\end{equation*}
where
\begin{equation} \label{P'_even}
\ P_{\operatorname{even}}'(z) 
:= \
2^{-2r+1}P_\Lambda (z^2)\sum_{j =0}^r\binom{2r}{2j}z^{-2j}
\end{equation}
In a similar way we obtain
\begin{equation*}
Q_{k-1}^{\operatorname{odd}}(f)
\  = \
\sum_{s \in I(k-1)} T_{h^{(k)}}^{[P_{\operatorname{odd}}']}(f)((2s+1)h^{(k)}) N_{k,2s+1}, 
\end{equation*}
where
\begin{equation} \label{P'_odd}
\ P_{\operatorname{odd}}'(z) 
:= \
2^{-2r+1}P_\Lambda (z^2)\sum_{j =0}^{r-1}\binom{2r}{2j+1}z^{-2j-1}.
\end{equation}
We define
\begin{equation} \label{P_even,P_odd}
\ P_{\operatorname{even}}  
:= \
P_\Lambda - P_{\operatorname{even}}', \quad
\ P_{\operatorname{odd}}  
:= \
P_\Lambda - P_{\operatorname{odd}}'
\end{equation}
Then from the definition $q_k(f) = Q_k(f) - Q_{k-1}(f)$ we receive the following representation for $q_k(f)$,
\begin{equation} \label{eq[q_0(f)]}
q_0(f) 
 \ = \ 
\sum_{s \in I(0)} T_{h^{(0)}}^{[P_\Lambda]}(f)(sh^{(0)})) N_{0,s}, 
 \end{equation}
 and for $k>0$,
\begin{equation} \label{eq[q_k(f)]}
q_k(f) 
 \ = \ 
q_k^{\operatorname{even}}(f) + q_k^{\operatorname{odd}}(f)
 \end{equation}
 with
\begin{equation} \nonumber
\begin{aligned}
q_k^{\operatorname{even}}(f)
&= \
\sum_{s \in I(k-1)} T_{h^{(k)}}^{[P_{\operatorname{even}}]}(f)(2sh^{(k)}) N_{k,2s},\\[1.5ex]
q_k^{\operatorname{odd}}(f)
&= \ 
\sum_{s \in I(k-1)} T_{h^{(k)}}^{[P_{\operatorname{odd}}]}(f)((2s+1)h^{(k)}) N_{k,2s+1}.
\end{aligned}
\end{equation}

From the definitions of $Q_k$ and $q_k$ it follows that 
\[T_{h^{(k)}}^{[P_{\operatorname{even}}]}(g)(2sh^{(k)}) = 0 \quad \text{and} \quad  
T_{h^{(k)}}^{[P_{\operatorname{odd}}]}(g)((2s+1)h^{(k)}) = 0 \quad \text{for every} \quad g \in \Pp_r.
\] 
Hence, by 
Lemma \ref{lemma[factor]} we prove the following lemma for the univariate operators $q_k$.

\begin{lemma} \label{lemma[q_k^even&odd]}
We have
\begin{equation} \label{eq[q_k^even&odd]}
\begin{aligned}
P_{\operatorname{even}}(z)
&= \
D_{2r}(z) P_{\operatorname{even}}^*(z)\,\\[1.5ex]
P_{\operatorname{odd}}(z)
&= \ 
D_{2r}(z) P_{\operatorname{odd}}^*(z). 
\end{aligned}
\end{equation}
where $P_{\operatorname{even}}^*$, $P_{\operatorname{odd}}^*$ are a symmetric Laurent polynomial.
Therefore, in the representation \eqref{eq[q_0(f)]}--\eqref{eq[q_k(f)]} of $q_k(f)$, we have  for $k>0$,
\begin{equation} \nonumber
\begin{aligned}
q_k^{\operatorname{even}}(f)
&= \
\sum_{s \in I(k-1)} T_{h^{(k)}}^{[P_{\operatorname{even}}^*]}\circ 
\Delta_{h^{(k)}}^{2r} (f)(2sh^{(k)}) N_{k,2s},\\[1.5ex]
q_k^{\operatorname{odd}}(f)
&= \ 
\sum_{s \in I(k-1)} T_{h^{(k)}}^{[P_{\operatorname{odd}}^*]}\circ 
\Delta_{h^{(k)}}^{2r} (f)((2s+1)h^{(k)}) N_{k,2s+1}.
\end{aligned}
\end{equation}
Equivalently, in the representation \eqref{eq:RepresentationMixedq_k(f)} of $q_k(f)$, we have for $s \in I(0)$
\begin{equation} \nonumber
c_{0,s}(f) 
 \ = \ 
T_{h^{(0)}}^{[P_\Lambda]}(f)(sh^{(0)}),
 \end{equation}
 and for $k>0$ and $s \in I(k)$,
 \begin{equation} \nonumber
c_{k,s}(f) 
 \ = \
\begin{cases}
T_{h^{(k)}}^{[P_{\operatorname{even}}^*]}\circ \Delta_{h^{(k)}}^{2r} (f)(sh^{(k)}),  & s \ \text{even} \\[1.5ex]
T_{h^{(k)}}^{[P_{\operatorname{odd}}^*]}\circ \Delta_{h^{(k)}}^{2r} (f)(sh^{(k)}), & s \ \text{odd}.
\end{cases} 
\end{equation}
\end{lemma}

\begin{proof} Consider the representation \eqref{eq:RepresentationMixedq_k(f)}  for $q_k(f)$ and $d=1$. If $g$ is arbitrary polynomial of degree at most $2r-1$, then since $Q_k$ reproduces $g$ we have $q_k(g) = 0$ and consequently, 
$c_{k,s}(g)=0$ for $k > 0$. The equations \eqref{eq[q_0(f)]}--\eqref{eq[q_k(f)]} give an explicit formula for the coefficient $c_{k,s}(g)$ as $T_{h^{(k)}}^{[P_{\operatorname{even}}]}(g)(2sh^{(k)})$ and $T_{h^{(k)}}^{[P_{\operatorname{odd}}]}(g)((2s+1)h^{(k)})$. Hence, by Lemma \ref{lemma[factor]} we get 
\eqref{eq[q_k^even&odd]}.
\end{proof}

\begin{theorem} \label{theorem[c_{k,s}]}
In the representation \eqref{eq:RepresentationMixedq_k(f)} of $q_k(f)$, we have for every 
$k \in \ZZdp(u)$ and $s \in I^d(k)$,
\begin{equation} \label{eq[c_{k,s}(f)](d>1)}
c_{k,s}(f) 
 \ = \ 
T_{h^{(k)}}^{[P_{k,s}]}(f)(sh^{(k)}), 
\end{equation}
where
\begin{equation} \label{eq[P_{k,s}]}
P_{k,s} (z)
 \ = \
\prod_{j \not\in u} P_\Lambda(z_j)\,\prod_{j \in u} P_{k_j,s_j}^*(z_j) \prod_{j \in u} D_{2r}(z_j), 
\end{equation}
\begin{equation} \label{eq[P_{k,s}^*]}
P_{k_j,s_j}^*(z_j) 
 \ = \
\begin{cases}
P_{\operatorname{even}}^*(z_j),  & s \ \text{even}, \\[1.5ex]
P_{\operatorname{odd}}^*(z_j), & s \ \text{odd}.
\end{cases} 
\end{equation}
\end{theorem}

\begin{proof} Indeed, from the definition of $c_{k,s}(f)$ and Lemma \ref{lemma[q_k^even&odd]} we have
for every $k \in \ZZdp(u)$ and $s \in I^d(k)$,
\begin{equation} \nonumber
c_{k,s}(f) 
 \ = \ 
\biggl(\prod_{j=1}^d T_{h^{(k)}_j}^{[P_{k_j,s_j}]}\biggl)(f)(sh^{(k)}), 
 \ = \ 
T_{h^{(k)}}^{[P_{k,s}]}(f)(sh^{(k)}),
\end{equation}
where
\begin{equation} \nonumber
P_{k_j,s_j}(z_j) 
 \ = \
\begin{cases}
P_\Lambda(z_j), \ & k_j=0, \\[1.5ex] 
P_{\operatorname{even}}^*(z_j) D_{2r}(z_j), \ & k_j > 0, \ s \ \text{even} \\[1.5ex]
P_{\operatorname{odd}}^*(z_j) D_{2r}(z_j), \ & k_j > 0,\ s \ \text{odd}.
\end{cases} 
\end{equation}
\end{proof}

For a Laurent polynomial $P$ as in \eqref{def[P]} we introduce the norm
\begin{equation} \nonumber
\|P\|
:= \ 
\sum_{s \in A}  |c_s|.
\end{equation}
Notice that
$
\|L_\Lambda\|_{C(\TT)}
\ \le \
\|P_\Lambda\|.
$

For given $\Lambda$, $1 \le p \le \infty$ and $0 < \alpha \le 2r$, we define 
\begin{equation}  \label{[ab]}
\begin{split}
a \ &= \ a(p, \alpha, \Lambda, r):= \ (2r)^{- \alpha} [r(p+1)]^{-1/p} \, 
\max\big\{\|P_{\operatorname{even}}^*\|, \|P_{\operatorname{odd}}^*\|\big\}, \\[1.5ex]
b \ &= \ b(\Lambda)\ := \ \|L_\Lambda\|_{C(\TT).}
\end{split}
\end{equation} 

\begin{theorem} \label{theorem[Q-Representation-U(f)]}
Let $1 \le p \le \infty$, and $0 < \alpha \le 2r$. Let $f \in \Ha$. Then $f$ can be represented by  the series \eqref{eq:B-splineRepresentation} converging in the norm of $C(\TTd)$.
Moreover, we have for every $k \in \ZZdpu$,
\begin{equation}  \label{ineq[|q_k(f)|_p](2)}
\|q_k(f) \|_p
 \ \le \ 
 a^{|u|}\, b^{d-|u|}\, 2^{- \alpha|k|_1} \, |f|_{\Hau}.
\end{equation} 
\end{theorem}

\begin{proof} 
The first part of the lemma on representation and convergence is in Lemma \ref{lemma[representation]}. Let us first prove  \eqref{ineq[|q_k(f)|_p](2)} for $p=\infty$. For convenience, we temporarily use the notation $a = a_p$. In this case, \eqref{ineq[|q_k(f)|_p](2)} is as 
\begin{equation}  \label{ineq[|q_k(f)|_infty]}
\|q_k(f) \|_\infty
 \ \le \ 
\sup_{s \in I^d(k)} \, |c_{k,s}(f)|  
 \ \le \ 
a_\infty^{|u|}\, b^{d-|u|}\, 2^{- \alpha|k|_1} \, |f|_{\Hau}.
\end{equation}
By definition we have 
\begin{equation*} 
q_k(f) 
 \ = \ 
\biggl(\prod_{j \not\in u} q_{k_j}\biggl)(g), \quad g:= \biggl(\prod_{j \in u} q_{k_j}\biggl)(f). 
\end{equation*}
Hence, by \eqref{[|Q_k(f)|<(T)]} we derive
 \begin{equation}  \label{ineq[|q_k(f)|_infty](2)}
\|q_k(f) \|_\infty
 \ \le \ 
b^{d-|u|}\, \|g\|_\infty,
\end{equation} 
Similarly to the proof of Theorem \ref{theorem[FaberRepresentation]} we can show that
for every $k \in \ZZdpu$,
\begin{equation} \label{ineq[|g|_infty]}
\|g \|_\infty
 \ = \ 
\biggl\|\biggl(\prod_{j \in u} q_{k_j}\biggl)(f)\biggl\|_\infty
\ \le \
\sup_{s \in I^d(k)} \, \biggl|\biggl(\prod_{j \in u} c_{k_j,s_j}\biggl)(f)\biggl|.   
\end{equation} 
Observe that for every Laurent polynomial and every $h \in \RR$,
\begin{equation} \label{[|T_h^{[P]}(f)|_infty<]}
\|T_h^{[P]}(f)\|_\infty
\ \le \
\|P\|\, \|f\|_\infty.
\end{equation}
Setting 
\begin{equation} \nonumber
P_{k,s}'(z) 
:= \
\prod_{j \in u} P_{k_j,s_j}^*(z_j), 
\end{equation}
we get
\begin{equation} \nonumber
\biggl(\prod_{j \in u} c_{k_j,s_j}\biggl)(f)
\ = \
T_{h^{(k)}}^{[P_{k,s}']}\biggl[\Delta^{l,u}_{h^{(k)}}(f,sh^{(k)})\biggl] 
 \end{equation}
Hence, by \eqref{eq[c_{k,s}(f)](d>1)}--\eqref{eq[P_{k,s}^*]} and \eqref{[|T_h^{[P]}(f)|_infty<]}
\begin{equation*} 
\begin{aligned}
\sup_{s \in I^d(k)} \, \biggl|\biggl(\prod_{j \in u} c_{k_j,s_j}\biggl)(f)\biggl|
 \ &= \ 
\sup_{s \in I^d(k)} \biggl|T_{h^{(k)}}^{[P_{k,s}']}\biggl[\Delta^{l,u}_{h^{(k)}}(f,sh^{(k)}))\biggl]\biggl| \\[1.5ex]
 \ &\le \ 
 \sup_{s \in I^d(k)} \|P_{k,s}'\|\biggl|\Delta^{l,u}_{h^{(k)}}(f,sh^{(k)}))\biggl| \\[1.5ex]
 \ &\le \ 
\big[\max\big\{\|P_{\operatorname{even}}^*\|, \|P_{\operatorname{odd}}^*\|\big\}\big]^{|u|}\,
\prod_{j \in u}|h^{(k)}_j|^\alpha \, |f|_{\Hau}\\[1.5ex]
  \ &\le \ 
a_\infty^{|u|}\, 2^{- \alpha|k|_1} \, |f|_{\Hau}.
\end{aligned}
\end{equation*}
This together with \eqref{ineq[|q_k(f)|_infty](2)} and \eqref{ineq[|g|_infty]} proves \eqref{ineq[|q_k(f)|_infty]}.

Let us prove \eqref{ineq[|q_k(f)|_p](2)} for the case $1 \le p < \infty$. We have for every $k \in \ZZdpu$,
\begin{equation*} 
\begin{aligned}
\|q_k(f) \|_p^p
 \ &= \ 
 \int_{\TTd}\Big| \sum_{s \in I^d(k)} \, c_{k,s}(f) M_{k,s} (x) \Big|^p \, dx  \\[1.5ex]
  \ &\le \ 
 \sup_{s \in I^d(k)} \, |c_{k,s}(f)|^p \, 
\int_{\TTd}\Big| \sum_{s \in I^d(k)} M_{k,s}(x) \Big|^p \, dx  \\[1.5ex]
 \ &= \ 
 \sup_{s \in I^d(k)} \, |c_{k,s}(f)|^p \, 
 |I^d(k)| \int_{\TTd}|M_{k,0}(x)|^p \, dx  \\[1.5ex]
 \ &= \ 
 \sup_{s \in I^d(k)} \, |c_{k,s}(f)|^p \, 
 |I^d(k)| \prod_{j \in u} \int_{0}^{(2r)^{-1}2^{-k_j}}|M(2^{k_j}x_j)|^p \, dx_j  \\[1.5ex]
 \ &= \ 
  \sup_{s \in I^d(k)} \, |c_{k,s}(f)|^p \, 
 |I^d(k)| \, (2r)^{-|u|}\, 2^{-|k|_1} \, \biggl( \int_{\RR} |M_{2r}(t)|^p \, dt \biggl)^{|u|}. 
\end{aligned}
\end{equation*}
By employing  \eqref{ineq[|q_k(f)|_infty]}, Young's inequality 
\begin{equation*} 
\|M_{2r}\|_{L_p(\RR)}= \|M_{2r-2} * M_2\|_{L_p(\RR)}\, \le \, \|M_{2r-2}\|_{L_1(\RR)}\|M_2\|_{L_p(\RR)},
\end{equation*}
the equations $\|M_{2r-2}\|_{L_1(\RR)}=1$ and
$
\|M_2\|_{L_p(\RR)}^p
\ = \
2(p+1)^{-1}, 
$
we complete the estimation as follows
\begin{equation*} 
\begin{aligned}
\|q_k(f) \|_p^p
 \ &\le \ 
\biggl[a_\infty^{|u|}\, b^{d-|u|}\, 2^{- \alpha|k|_1} \, |f|_{\Hau}\biggl]^p  \, 
 \big[2^{|k|_1}\big] \, \big[(2r)^{-|u|}\, 2^{-|k|_1}\big]\, \big[2^{|u|} (p+1)^{-|u|}\big] \\[1.5ex]
 \ &= \ 
 \big[r(p+1)\big]^{-|u|} \, \biggl[a_\infty^{|u|}\, b^{d-|u|} \biggl]^p  \, 
 2^{- p\alpha|k|_1} \, |f|_{\Hau}^p \\[1.5ex]
\ &= \ 
 \biggl[a_p^{|u|}\, b^{d-|u|} \biggl]^p  \, 
 2^{- p\alpha|k|_1} \, |f|_{\Hau}^p.
\end{aligned}
\end{equation*}
\end{proof}

Theorems on B-spline quasi-interpolation representations with discrete equivalent quasi-norm in terms of coefficient
functionals have been proved in \cite{Di09}--\cite{Di13} for non-periodic various Besov spaces.


\subsection{Sampling recovery and cubature}

For $m \in {\ZZ}_+$, we define the operator $R^\nu_m$ by 
\begin{equation*}
R^\nu_m(f) 
:= \ 
\sum_{k \in \ZZdp: \, |\supp(k)| \le \nu, \ |k|_1 \le m} q_k(f)
\ = \
\sum_{k \in \ZZdp: \, |\supp(k)| \le \nu, \ |k|_1 \le m} \ \sum_{s \in I^d(k)} c_{k,s}(f)\, N_{k,s}.
\end{equation*} 
For functions $f$ on $\TTd$ having at most $\nu$ of active variables, $R^\nu_m$ defines the linear sampling algorithm 
on the Smolyak grid $G^\nu(m)$ 
\begin{equation*} 
R^\nu_m(f) 
\ = \ 
S_n(\Psi^\nu_m,f) 
\ = \ 
\sum_{\xi \in G^\nu(m)} f(\xi) \psi_{\xi}, 
\end{equation*} 
where $n := \ |G^\nu(m)|$, $\Psi^\nu_m:= \{\psi_{\xi}\}_{\xi \in G^\nu(m)}$  and for $\xi =  2^{-k} s$,  $\psi_\xi$ are explicitly constructed as linear combinations of at most 
at most $N$ B-splines $N_{k,j}$ for some $N \in \NN$ which is independent of $k,s,m$ and $f$.  

With where $a,b$ are as in \eqref{[ab]} we put
\[
\delta(\nu,a,b,m)
:= \
 \sum_{l=0}^\nu \binom{\nu}{l}\, a^l \, b^{\nu-l} \, \beta(l,m), 
\]
and
\begin{equation} \nonumber
c(\nu,a,b)
:= \
|2^\alpha -2|^{- |\operatorname{sgn}(\alpha - 1)|} \times
\begin{cases}
[a + b]^\nu, \ & \alpha > 1,\\[1.5ex]
\nu\,[a + b]^\nu, \ & \alpha = 1,\\[1.5ex]
[a/(2^\alpha -1) +b]^\nu, \ & \alpha < 1.
\end{cases}
\end{equation}

\begin{theorem} \label{theorem[|f - R_m(f)|_p,r>1]}
Let $1 \le p \le \infty$, $0 < \alpha \le 2r$ and $1 \le \nu \le d$. Then we have for every $m \ge \nu$, 
\begin{equation}  \label{ineq[|f - R_m(f)|_p,r>1](1)}
\begin{aligned}
s^\nu_m(\Uan)_p
\ &\le \ 
\sup_{f \in \Uan} \|f - R^\nu_m(f)\|_p
\ \le \
\delta(\nu,a,b,m)\,  2^{- \alpha m}  \\[1ex]
\ &\le \  \exp(2^\alpha - 1) \, [a/(2^\alpha -1)+ b]^\nu \,  2^{- \alpha m}\, m^{\nu-1}.
\end{aligned}
\end{equation}
Moreover,  if in addition, $m \ge 2(\nu-1)$,
\begin{equation}  \label{ineq[|f - R_m(f)|_p,r>1](2)} 
s^\nu_m(\Uan)_p
\ \le \ 
\sup_{f \in \Uan} \|f - R^\nu_m(f)\|_p
\ \le \  
c(\nu,a,b) \,  2^{- \alpha m}\, \binom{m}{\nu-1}.
\end{equation}
\end{theorem}

\begin{proof} We prove the theorem for the case $\nu = d$ which yields the case $\nu < d$ by the same argument as in the proof of Theorem~\ref{theorem[|f - R_m(f)|_p,Has]}. Let $f \in \Ua$. Put $t=2^{- \alpha}$.
From Theorem \ref{theorem[Q-Representation-U(f)]} and  and \eqref{eq[BG]} it follows that for every $m \ge d - 1$, 
\begin{equation} \label{estimation[|f - R_m(f)|_p,r>1]}
\begin{aligned}
\|f - R^d_m(f)\|_p
  \ &\le \ 
 \sum_{|u|\le d} \ \sum_{k \in \ZZdpu: \, |k|_1 > m}\|q_k(f) \|_p \\[1.5ex]
  \ &\le \ 
 \sum_{|u|\le d} \ \sum_{k \in \ZZdpu: \, |k|_1 > m}
a^{|u|} \, b^{d - |u|} \, 2^{- \alpha|k|_1}  \\[1.5ex]
 \ &= \ 
 \sum_{l=0}^d \binom{d}{l}\, a^l \, b^{d - |u|} \,
 \sum_{k \in \NN^l: \, |k|_1 > m}   2^{- \alpha|k|_1}  \\[1.5ex]
 \ &= \ 
 \, \sum_{l=0}^d \binom{d}{l}\,a^l \, b^{d-l} \,  \sum_{j=1}^\infty
 \binom{m+j-1}{l-1}\, 2^{- \alpha (m+j)}  \\[1.5ex]
\ &= \ 
 2^{- \alpha (m+1)}\sum_{l=0}^d \binom{d}{l}\,a^l \, b^{d-l} \, F_{m,l-1}(t)  \\[1.5ex]
\ &= \ 
 2^{- \alpha m} \delta(\alpha,d,a,b,m).
 \end{aligned}
\end{equation}
The first inequality in \eqref{ineq[|f - R_m(f)|_p,r>1](1)} is proven.

Hence, applying  Lemma \ref{lemma[sum(2)]} for $t=2^{- \alpha}$ gives
\begin{equation} \nonumber
\begin{aligned}
\|f - R^d_m(f)\|_p
\ &\le \ 
\exp(2^\alpha - 1)\,  2^{- \alpha m} \sum_{l=0}^d \binom{d}{l}\, 
a^l \, b^{d-l} \, [2^\alpha-1]^{-l} m^{l-1} \\[1.5ex]
\ &\le \ 
\exp(2^\alpha - 1) \, [a/(2^\alpha -1)+ b]^d \, 2^{- \alpha m}\, m^{d-1}
\end{aligned}
\end{equation}
which proves the second inequality in \eqref{ineq[|f - R_m(f)|_p,r>1](1)}.

We now prove \eqref{ineq[|f - R_m(f)|_p,r>1](2)}. If $m \ge 2(d - 1)$, by \eqref{estimation[|f - R_m(f)|_p,r>1]} and Lemma \ref{lemma[sum(1)]} for $t=2^{- \alpha}$ we get
\begin{equation} \nonumber
\begin{aligned}
\|f - R^d_m(f)\|_p
\ &\le \ 
2^{- \alpha (m+1)}\sum_{l=0}^d \binom{d}{l}\, 
a^l \, b^{d-l} \,  \binom{m}{l-1}\, b_{l-1}(2^{-\alpha})  \\[1.5ex]
\ &\le \ 
 2^{- \alpha (m+1)} \binom{m}{d-1}\sum_{l=0}^d \binom{d}{l}\, 
a^l \, b^{d-l} \, b_{l-1}(2^{-\alpha})  \\[1.5ex]
\ &= \
c(\alpha, d,a,b)\, 2^{- \alpha m}\, \binom{m}{d-1}.
\end{aligned}
\end{equation}
The proof is complete.
\end{proof}


In a similar way to Theorem~\ref{theorem[Int_m<Has]}, from Theorem \ref{theorem[|f - R_m(f)|_p,r>1]} we obtain

\begin{theorem} \label{theorem[Int<,r>1]}
Let $0 < \alpha \le 2r$ and $1 \le \nu \le d$. Let 
$a = (2r)^{- \alpha-1} \, \max\big\{\|P_{\operatorname{even}}^*\|, \|P_{\operatorname{odd}}^*\|\big\}$ and 
$b= \|L_{\Lambda}\|_{C(\TT)}$ (cf. \eqref{[ab]}). Then we have for every $m \ge \nu$, 
\begin{equation}  \nonumber
\Int_m^\nu(\Uan) 
\ \le \ 
\delta(\nu,a,b,m)\,  2^{- \alpha m} 
\ \le \  \exp(2^\alpha - 1) \, [a/(2^\alpha -1)+ b]^\nu \,  2^{- \alpha m}\, m^{\nu-1}.
\end{equation}
Moreover,  if in addition, $m \ge 2(\nu-1)$,
\begin{equation}  \nonumber 
\Int_m^\nu(\Uan) 
\ \le \  
c(\nu,a,b) \, 2^{- \alpha m}\, \binom{m}{\nu-1}.
\end{equation}
\end{theorem}

\section{Examples}
\label{Examples}

The upper bounds  for $s^\nu_m(\Uan)_p$ and $\Int_m^\nu(\Uan)$ obtained in Theorem~\ref{theorem[|f - R_m(f)|_p,r>1]} and Theorem~\ref{theorem[Int<,r>1]} depend on the parameters $p, \alpha, d, \nu, r,a,b$. In applications, if the 
parameters $p, \alpha, d, \nu$ are a priori known, these upper bounds are controlled by parameters $r,a,b$ which are determined by the choice of a univariate quasi-interpolation operator $Q$ of the form  \eqref{def:Q}. 
The operator $Q$ is induced by the sequence $\Lambda$  as in \eqref{def:Lambda} which can be uniquely characterized by the univariate symmetric Laurent polynomial
$P_\Lambda$. Moreover, the parameters $a,b$ defined in \eqref{[ab]} contain  
$\max\big\{\|P_{\operatorname{even}}^*\|, \|P_{\operatorname{odd}}^*\|\big\}$ 
and $\|L_{\Lambda}\|_{C(\TT)}$ which are desirable to be minimum by the choice of $P_{\Lambda}$. 
In this subsection, we give some examples of the univariate symmetric Laurent polynomial $P_\Lambda$ characterizing quasi-interpolation operator $Q$ of the form  \eqref{def:Q} with approximate estimates of  
$\|P_{\operatorname{even}}^*\|$,  $\|P_{\operatorname{odd}}^*\|$ and $\|L_\Lambda\|_{C(\TT)}$ or 
$\|P_\Lambda\|$ with $\|L_\Lambda\|_{C(\TT)} \le \|P_\Lambda\|$. For a given $P_\Lambda$, the Laurent polynomials 
$P_{\operatorname{even}}^*$ and  $P_{\operatorname{odd}}^*$ can be computed from \eqref{P'_even}--\eqref{P_even,P_odd} and \eqref{eq[q_k^even&odd]}.

\subsection{Piece-wise linear quasi interpolation}
Let us consider the case $r=1$ when $M(x)\ = \ (1 - |x-1|)_+$ is the piece-wise linear cardinal B-spine with knot at $0,1,2$. Let $\Lambda = \{\lambda(s)\}_{j=0}$ $(\mu=0)$ be a given by $\lambda(0)= 1$. 
If $N_k$ is the periodic extension of $M(2^{k+1}\cdot)$,  then 
\begin{equation*}
N_{k,s}(x):= \ N_k(x - s), \ k \in {\ZZ}_+, \ s \in I(k), 
\end{equation*}
where $I(k) := \{0,1,..., 2^{k+1} - 1\}$.
Consider the related periodic quasi-interpolation operator for functions $f$ on $\TT$ and $k \in \ZZ_+$,
\begin{equation} \nonumber
Q_k(f,x)= \ \sum_{s \in I(k)} f(2^{-(k+1)}(s+1)) N_{k,s}(x) 
\end{equation}
We have
\begin{equation} \nonumber
\begin{aligned}
P_\Lambda(z)
\ &= \ 
z, \\[1.5ex]
P_{\operatorname{even}}(z)
\ &= \
- \frac{1}{2} (z-1)^2, \quad 
P_{\operatorname{even}}^*(z)
\ = \
- \frac{1}{2},\\[1.5ex]
P_{\operatorname{odd}}(z)
&= \ 
P_{\operatorname{odd}}^*(z)
\ = \ 
0,
\end{aligned}
\end{equation}
and
\begin{equation} \nonumber
\|L_\Lambda\|
\ = \ 
1, \quad
\|P_{\operatorname{even}}^*\|
\ = \ 
\frac{1}{2}, \quad
\|P_{\operatorname{odd}}^*\|
\ = \ 
0.
\end{equation}
Hence,
\begin{equation} \nonumber
q_0(f) 
 \ = \ 
\sum_{s=0}^1 T_{2^{-1}}^{[P_\Lambda]}(f)(2^{-1}s) N_{0,s}
\ = \
f(0), 
 \end{equation}
 and for $k>0$,
\begin{equation} \nonumber
q_k(f) 
 \ = \ 
q_k^{\operatorname{even}}(f) 
\ = \
\sum_{s =0}^{2^k-1} 
\biggl\{- \frac {1}{2} \Delta_{2^{-(k+1)}}^2 f(2^{-k}s)\biggl\} N_{k,2s}.
\end{equation}
With these formulas for $q_k(f)$ after redefining $N_{k,2s}$ as $\varphi_{k,s}$, the quasi-interpolation representation 
\eqref{eq:B-splineRepresentation} becomes the Faber series.

\subsection{Cubic B-spline quasi-interpolation}

For $r=2$, we can take   
\begin{equation} \nonumber
P_\Lambda(z)
\ = \ 
\frac{z^2}{6}(- z + 8 - z^{-1}),
\ = \ 
- \frac{1}{6} z^3 + \frac{8}{6} z^2 - \frac{1}{6} z.
\end{equation}
Then, we have
\begin{equation} \nonumber
\begin{aligned}
P_{\operatorname{even}}(z)
\ &= \
(z-1)^4 P_{\operatorname{even}}^*(z), \quad 
P_{\operatorname{even}}^*(z)
\ = \
\frac{1}{48}z^{-2}\biggl( z^4 + 4z^3 + 8z^2 + 4z + 1 \biggl),\\[1ex]
P_{\operatorname{odd}}(z)
&:= \ 
(z-1)^4 P_{\operatorname{odd}}^*(z), \quad
P_{\operatorname{odd}}^*(z)
:= \ 
\frac{1}{12} \biggl( z^2 + 4z + 1 \biggl),
\end{aligned}
\end{equation}
and 
\begin{equation} \nonumber
\|L_\Lambda\|
\ = \ 
\frac{11}{9}, \quad
\|P_{\operatorname{even}}^*\|
\ = \ 
\frac{3}{8}, \quad
\|P_{\operatorname{odd}}^*\|
\ = \ 
\frac{1}{2}.
\end{equation}

\subsection{Quintic B-spline quasi-interpolation representation}
For $r=3$,we can take 
\begin{equation} \nonumber
P_\Lambda(z)
:= \ 
\frac{z^3}{14400} [25150 - 5876(z + z^{-1}) + 448(z^2 + z^{-2}) + 52(z^3 + z^{-3}) + (z^4 + z^{-4})].
\end{equation}
Then, we have
\begin{equation} \nonumber
P_{\operatorname{even}}(z)
\ = \
(z-1)^6 P_{\operatorname{even}}^*(z), \quad
P_{\operatorname{odd}}(z)
\ = \ 
(z-1)^6 P_{\operatorname{odd}}^*(z),
\end{equation}
\begin{equation} \nonumber
\begin{aligned}
P_{\operatorname{even}}^*(z)
\ &= \
\frac{1}{460800} \big[139760 + 97002(z + z^{-1}) + 42508(z^{2} + z^{-2}) + 11462(z^{3} + z^{-3}) \\[1.5ex]
\ & \quad + \ 2328(z^{4} + z^{-4})  + 458(z^{5} + z^{-5}) + 36(z^{6} + z^{-6}) + 6(z^{7} + z^{-7}) \big],\\[1.5ex]
P_{\operatorname{odd}}^*(z)
&= \ 
\frac{z}{460800} \big[164910 + 97002(z + z^{-1}) + 36632(z^{2} + z^{-2}) + 11462(z^{3} + z^{-3}) \\[1.5ex]
\ & \quad + \ 2776(z^{4} + z^{-4})  + 458(z^{5} + z^{-5}) + 88(z^{6} + z^{-6}) + 6(z^{7} + z^{-7}) + (z^{8} + z^{-8})\big].\\[1.5ex]
\end{aligned}
\end{equation}
and
\begin{equation} \nonumber
\|P_\Lambda\|
\ = \ 
\frac{37904}{14400}
\ \approx \ 2.63, \quad
\|P_{\operatorname{even}}^*\|
\ = \ 
\frac{447360}{460800}
\ \approx \ 0.97, \quad
\|P_{\operatorname{odd}}^*\|
\ = \ 
\frac{467160}{460800}
\ \approx \ 1.00.
\end{equation}

\noindent
{\bf Acknowledgments.}  This work is funded by Vietnam National Foundation for Science and Technology Development (NAFOSTED) under  Grant No. 102.01-2014.02. A part of this work was done when the author was working as a research professor at the Vietnam Institute for Advanced Study in Mathematics (VIASM). He  would like to thank  the VIASM  for providing a fruitful research environment and working condition.

\end{document}